\documentclass[review,onefignum,onetabnum]{siamart190516}

\usepackage{lipsum}
\usepackage{amsfonts}
\usepackage{graphicx}
\usepackage{epstopdf}
\usepackage{algorithmic}

\usepackage{graphicx}          
\usepackage{amssymb}
\usepackage{changepage}

\usepackage{amsmath}
\usepackage{graphicx}
\usepackage[colorinlistoftodos]{todonotes}

\usepackage{enumitem}

\usepackage{amsmath,verbatim,amssymb,amsfonts,amscd}
\usepackage{mathabx}
\usepackage{breqn}
\usepackage{float}
\usepackage{graphics}
\usepackage{stmaryrd} 
\usepackage{textcomp}
\usepackage{setspace}

\usepackage{textgreek}

\usepackage{enumitem}   
\usepackage{titlesec}

\usepackage{kantlipsum}
\ifpdf
  \DeclareGraphicsExtensions{.eps,.pdf,.png,.jpg}
\else
  \DeclareGraphicsExtensions{.eps}
\fi

\newcommand{\diff}{\mathrm{d}}

\newsiamremark{remark}{Remark}
\newsiamremark{hypothesis}{Hypothesis}
\crefname{hypothesis}{Hypothesis}{Hypotheses}
\newsiamthm{claim}{Claim}

\headers{Non-uniform Observability for MHE and stability }{E. Flayac and I. Shames}

\title{Non-uniform Observability for Moving Horizon Estimation and stability with respect to additive perturbation}

\author{Emilien  Flayac \thanks{Complex System Engineering Department, ISAE-Supaero, 31000 Toulouse, France.
  (\email{emilien.flayac@isae-supaero.fr}). Emilien Flayac was a postdoctoral fellow at the University of Melbourne when the reseach of this paper was developped}
\and Iman Shames\thanks{CIICADA Lab, School of Engineering,  Australian National University, Acton ACT 0200, Australia.
  (\email{iman.shames@anu.edu.au}).}}

\usepackage{amsopn}

\usepackage
[
        a4paper,
        left=2.5cm,
        right=3.7cm,
        top=3cm,
        bottom=4cm,
]
{geometry}

\ifpdf
\hypersetup{
  pdftitle={Non-uniform Observability for Moving Horizon Estimation and stability with respect to additive perturbation},
  pdfauthor={E. Flayac and I. Shames}
}
\fi

\nolinenumbers
\begin{document}

\maketitle

\begin{abstract}

 This paper formalises the concepts of weakly and weakly regularly persistent input trajectory as well as their link to the Observability Grammian and the existence and uniqueness of solutions of Moving Horizon Estimation (MHE) problems. Additionally, thanks to a new time-uniform Implicit Function Theorem, these notions are proved to imply the stability of MHE solutions with respect to small additive perturbation in the measurements and in the dynamics, both uniformly and non-uniformly in time. Finally, examples and counter-examples of weakly persistent and  weakly regularly persistent input trajectories are given in the case of 2D bearing-only navigation.

\end{abstract}



\section{Introduction}
In tackling nonlinear estimation problems using the machinery of mathematical optimisation, two ideas prevail. The most straightforward one is to define a cost on the complete sequence of past inputs and outputs and to estimate the associated state trajectory by minimising that cost over state trajectories. The estimator is then built from the resulting optimal state trajectory. This leads to Full Information Estimation (FIE). To reduce the computational cost and memory usage, another idea is to use a truncated version of the input/output sequence on a time window of fixed length and to keep the optimal state trajectories on this moving horizon. This leads to Moving Horizon Estimation (MHE). See Chapter 4 of \cite{rawlings_model_2020} for a general survey on these techniques.
In the classical literature on FIE and MHE, robust stability of the estimation error is usually proved under observability or detectability assumptions.  For example, in \cite{alessandri_moving-horizon_2008,michalska_moving_1995,rao_constrained_2003}, the stability of MHE schemes has been shown by assuming the so-called \emph{$N$-step observability} property. This assumption means that on a moving time window in a discrete-time framework, small errors between output trajectories must imply small errors in the initial states,  for any pair of initial states and  uniformly  with respect to the control input.  In \cite{allan_robust_2021, hu_robust_2017,  knuefer_nonlinear_2021,knufer_time-discounted_2020,muller_nonlinear_2017}, the FIE and MHE estimators are proved to be Robustly Globally Asymptotically Stable  under several versions of \emph{incremental input/output-to-state stability} (i-IOSS).  It can be interpreted as a robust detectability condition of any initial conditions in the presence of process noise,  measurement noise and/or control input. Note that in the above mentioned works, the comparison functions used to characterise the  i-IOSS property are again independent of any control input which means that uniform detectability is assumed. Global stability of classical FIE and MHE schemes require global solutions of the optimal estimation problem which may not be achievable in a general nonlinear case. This remark has notably been made in \cite{alamir_optimization_1999,alamir_further_2002,alessandri_fast_2017,morari_efficient_2009,kang_moving_2006,wynn_convergence_2014} where one only searches for state trajectories that are locally optimal.
%
%
A direct consequence of this restriction is that one does not need to be able to distinguish all the states from each other but only those close to the current state.  This means that the required observability conditions can be weakened accordingly. For instance, in \cite{wynn_convergence_2014}, a version of the $N$-step observability property localised around the actual state of the system is used to show the convergence of an approximate MHE scheme. These weaker assumptions are again made uniformly with respect to the control input. This suggests that the impact of the input trajectory on the  performance of the MHE scheme is overlooked. Nevertheless, it is known that general nonlinear observability properties of nonlinear controlled systems cannot be stated independently of the input, see \cite{besancon_nonlinear_2007}. In particular, some input trajectories might prevent the system from satisfying the $N$-step observability property. In this regard, the notion of regularly persistent input trajectories happens to be very useful, particularly, in the design of global observers for state-affine systems, see \cite{besancon_nonlinear_2007}. It defines a class of input trajectories in a continuous time framework that forces the system to satisfy the equivalent of the $N$-step observability property on the whole statespace. However, this property is so strong that such input trajectories might not exist. It is also unnecessary in many applications of MHE, as mentioned before. That is why, the first two contributions of this paper are to bring to light the links between classical nonlinear observability concepts and the problems of FIE and MHE and to formalise and characterise the new concepts of \emph{weakly} and \emph{weakly regularly persistent} input trajectories using the Observability Grammian.
They are written in the language of classical nonlinear observability theory in continuous time and provide a new framework for the study of existence, local uniqueness and stability of local solutions of MHE problems. More precisely, as the third and main contribution of this paper, we show that weakly and weakly regularly persistent input trajectories ensure that MHE problems still have locally unique local solutions close to the true state in the presence of small arbitrary bounded additive perturbation in the measurements and  in the dynamics. These results involve a new time-uniform Implicit Function Theorem in Banach spaces. Finally, we provide examples and counter-examples of weakly and weakly regularly persistent  input trajectories for a two-dimensional bearing-only system. In particular, we show that there exist weakly persistent input trajectories that do no satisfy our sufficient conditions for  weak regular persistence based on the Observability Grammian.

The rest of the paper is organized as follows. In Section \ref{sec:observability}, the standard nonlinear observability concepts are recalled. In Section \ref{sec:obs_opti}, explicit connections between these observability notions and optimisation concepts are established. In Section \ref{sec:weak_per}, the notions of weakly and weakly regularly persistent input trajectories are introduced and characterised through the Observability Grammian and the stability of the solution of a perturbed MHE problem based on a implicit function theorem for sequences of solutions of a smooth nonlinear equation is also proved. Finally, in Section \ref{sec:example_loc}, examples of weakly and weakly regularly persistent input trajectories are given in the case of bearing-only localisation in order to demonstrate the relevance of the proposed observabilty notions.


\section{Observability properties of general nonlinear controlled systems}\label{sec:observability}
This section is dedicated to the presentation of classical nonlinear observability concepts.
\subsection{Setup and classical nonlinear observability notions} \label{sec_class_nonlinear_obs}
To begin with, several well-known observability concepts are recalled from \cite{besancon_nonlinear_2007}. In the following, we denote by $\mathbb{N}$ the set of positive integers and by $\mathbb{R}^+$ the set of non-negative real numbers. We fix $(n_x,n_u,n_y)\in \mathbb{N}^3$.  We consider the following general nonlinear system:
   \begin{align}
       \dot{x}&=f(x,u), \label{eq:general_dyn_continous_time}\\
        y&=h(x,u)\notag,
   \end{align}
   where
   \begin{itemize}
   \item $u:\mathbb{R}^+ \longrightarrow U\subset \mathbb{R}^{n_u}$ is a piece-wise continuous input trajectory, $x$ is the corresponding state trajectory valued in $\mathbb{R}^{n_x}$ and $y$ the corresponding measurement (or output) trajectory valued in $\mathbb{R}^{n_y}$;
    \item  $f: \mathbb{R}^{n_x}\times \mathbb{R}^{n_u} \longrightarrow \mathbb{R}^{n_x}$ is the controlled vector field of the system and $h: \mathbb{R}^{n_x}\times \mathbb{R}^{n_u}\longrightarrow \mathbb{R}^{n_y}$ is the observation function, also called  output function. Mappings $f$ and $h$ are both assumed to be twice continuously differentiable.
   \end{itemize}
For simplicity, the solutions of system  \eqref{eq:general_dyn_continous_time} are supposed to be uniquely defined at all times. For $s_2\geq s_1\geq 0$, and $\xi \in \mathbb{R}^{n_x}$, we denote by $\phi_f(s_2;s_1,\xi,u)$ the solution flow of system \eqref{eq:general_dyn_continous_time} at time $s_2$ with initial condition $\xi$, initial time $s_1$ and input trajectory $u$. Let $x_0 \in \mathbb{R}^{n_x}$ be a fixed initial condition and $t_0=0$ be the reference initial time. In the following, the reference trajectory is defined, for some input trajectory $u$, by:
\begin{align}
    x(t):=\phi_f(t;0,x_0,u).\label{eq:reference_traj}
\end{align}
The property of observability of a system is defined as one's ability to distinguish between two initial conditions using only an input trajectory and the corresponding output trajectories. The definitions of distinguishable and indistinguishable pairs are recalled in Definition \ref{def:distinguishability}.

\begin{definition}[Distinguishability] 
\label{def:distinguishability} Let $u$ be an input trajectory.
A pair $(\xi_1,\xi_2)\in \mathbb{R}^{n_x}\times\mathbb{R}^{n_x}$ is said to be \emph{distinguishable using the input trajectory $u$} if there exists $t\geq0$ such that:
\begin{align*}
    h(\phi(t,0,\xi_1,u),u(t))\neq h(\phi(t,0,\xi_2,u),u(t)).
\end{align*}
A pair $(\xi_1,\xi_2)$ is said to be \emph{distinguishable} if there exists an input trajectory $u$ such that $(\xi_1,\xi_2)$ is distinguishable using the input trajectory $u$. If  $(\xi_1,\xi_2)$ is distinguishable (resp. using input trajectory $u$) then it is also said that $\xi_1$ \emph{is distinguishable from} $\xi_2$  (resp. using input trajectory $u$). If $(\xi_1,\xi_2)$  is not distinguishable, then it is said to be \emph{indistinguishable}.
\end{definition}

Therefore, observable systems are such that every initial state can be distinguished from the other states.

\begin{definition}[Observability]\label{def:observability} System \eqref{eq:general_dyn_continous_time} is said to be \emph{observable at $x_0\in \mathbb{R}^{n_x}$ } if for any $\xi \in \mathbb{R}^{n_x}$, $\xi$ is distinguishable from $x_0$. System \eqref{eq:general_dyn_continous_time} is said to be \emph{observable}  if for any $(\xi_1,\xi_2)\in (\mathbb{R}^{n_x})^2$, the pair $(\xi_1,\xi_2)$ is distinguishable.
\end{definition}
Note that, contrary to linear systems, observability of nonlinear systems depends  on input trajectories. In fact, observability as defined in Definition \ref{def:observability} requires the existence of an input trajectory for any pair of states in the statespace, that enables one to discriminate them. This makes  observability a strong property that might not be satisfied by a large class of systems. This justifies the introduction of the concept of weak observability where one focuses on a neighbourhood of some state.
\begin{definition}[Weak observability]\label{def:weak_obs}The system \eqref{eq:general_dyn_continous_time} is said to be \emph{weakly observable at $x_0$} if there exists an input trajectory $u$ and a neighbourhood, $\mathbb{U}$, of $x_0$ such that for any $\xi \in \mathbb{U} \backslash \{x_0\}$, there exists $t\geq0$ such that:
   \begin{align*}
        h(\phi_f(t;0,x_0,u),u(t))&\neq h(\phi_f(t;0,\xi,u),u(t)).
   \end{align*}
 The system \eqref{eq:general_dyn_continous_time} is said to be \emph{weakly observable} if it is  weakly observable at $x_0$ for any $x_0 \in \mathbb{R}^{n_x}$.
   \end{definition}

A slightly stronger concept of observability is used when one also needs to distinguish a pair of states instantly that is to say by staying close to the initial condition. For this reason, the notion of local weak observability has been introduced in \cite{hermann_nonlinear_1977}. Its definition is recalled in Definition \ref{def:local_weak_obs}.

  \begin{definition}[Local weak observability]\label{def:local_weak_obs}
   The system \eqref{eq:general_dyn_continous_time} is said to be \emph{locally weakly observable at $x_0$} if there exists an input trajectory $u$ and a neighbourhood, $\mathbb{U}$, of $x_0$ such that for any neighbourhood, $\mathbb{V}\subset\mathbb{U}$, of $x_0$ and any $\xi \in \mathbb{V} \backslash \{x_0\}$, there exists $t\geq0$ such that:
   \begin{align*}
        h(\phi_f(t;0,x_0,u),u(t))&\neq h(\phi_f(t;0,\xi,u),u(t)),\\
         \phi_f(t;0,\xi,u)&\in \mathbb{V}.
   \end{align*}
   
     The system \eqref{eq:general_dyn_continous_time} is said to be \emph{locally weakly observable} if it is locally weakly observable at $x_0$ for any $x_0 \in \mathbb{R}^{n_x}$.
   \end{definition}
In Definition \ref{def:local_weak_obs}, the term `weak' specifically refers to the fact that one is trying to distinguish between states that are near $x_0$ while the term `local' means that one is able to use arbitrarily short state trajectories to do so. Thus, local weak observability at some initial condition $x_0$ means that $x_0$ can be distinguished from its neighbours using the input and output trajectories corresponding to state trajectories $x$ that stay close to $x_0$.  Its main interest is that it can be checked using a rank condition on the Lie derivatives of $h$ along the vector fields defined by $f$. See \cite{besancon_nonlinear_2007} for more details.

Note that in Definition \ref{def:observability}, \ref{def:weak_obs} and \ref{def:local_weak_obs}, an element of the statespace is fixed and one focuses on the existence of an input trajectory that allows one to distinguish this element from others.  There exists another take on observability where one fixes a control trajectory and wonders if it can be used to distinguish between every pair of states. Such input trajectories are called  universal input trajectories.

\begin{definition}[Universal input]
   \label{def:universal_input}
   For $t\geq0$, an input trajectory  $u$  is a \emph{universal input trajectory on $[0,t]$} if for any $  \xi_1\neq \xi_2$,  there exists $ s \in [0,t]$  such that
   \begin{align*}
       h(\phi_f(s;0,\xi_1,u),u(s))\neq h(\phi_f(s;0,\xi_2,u),u(s)).
   \end{align*}
   An input trajectory is said to be a \emph{universal input trajectory} if there exists $t\geq0$ such that it is  a universal input trajectory on $[0,t]$. System \eqref{eq:general_dyn_continous_time} is said to be \emph{uniformly observable} if all input trajectories are universal.
     \end{definition}
   In the following, we focus on integral formulations of observability as they typically provide more quantitative notions. This leads to the definition of the cumulative output error.
   
   \begin{definition}[Cumulative output error]\label{def:cum_error_obs}
    For $0\leq t_1\leq t_2$, an input trajectory $u$ and a pair of states $(\xi_1,\xi_2)$ we define the \emph{cumulative output  error} of system \eqref{eq:general_dyn_continous_time} on $[t_1,t_2]$ at $(\xi_1,\xi_2)$ with input trajectory $u$, denoted by $l(t_1,t_2,\xi_1,\xi_2,u)$, as follows:
    \begin{align*}
        &l(t_1,t_2,\xi_1,\xi_2,u)=\int_{t_1}^{t_2}{\Vert h(\phi_f(s;t_1,\xi_1,u),u(s))- h(\phi_f(s;t_1,\xi_2,u),u(s))\Vert}^2\mathrm{d}s,
    \end{align*}
    where $\Vert \cdot \Vert$ denotes the Euclidian norm.
   \end{definition}
   
  Thus, from Definition \ref{def:universal_input}, one can derive an equivalent integral characterization of universal input trajectories.
   \begin{proposition} \label{prop_universal_input}
   An input trajectory $u$ is universal
   if and only if for any $ \xi_1\neq \xi_2$, there exists $t\geq0$ such that: 
   \begin{align}
        l(0,t,\xi_1,\xi_2,u)>0.\label{eq:integral_universal_input}
   \end{align}
   \end{proposition}
   \begin{proof}
   Since $u$ is assumed to be piece-wise continuous and $h$ is continuous, for any $ \xi_1\neq \xi_2$ and $t\geq0$, $\int_0^t{\Vert h(\phi_f(s;0,\xi_1,u),u(s))- h(\phi_f(s;0,\xi_2,u),u(s))\Vert}^2\mathrm{d}s=0$ if and only if for any $s\in [0,t]$, $h(\phi_f(s;0,\xi_1,u),u(s))=h(\phi_f(s;0,\xi_2,u),u(s))$. The result follows from this.
   \end{proof}


In theory, when a universal input trajectory is available, it should be possible to reconstruct the state of the system at anytime if one waits for a sufficiently long time. However, in practice, one would like to know an upper bound on the time required to distinguish states using some input trajectory. We first recall the classical definition of $\mathcal{K}$-functions.
   \begin{definition}[ $\mathcal{K}$-function]
   A function $\kappa: \mathbb{R}^+\rightarrow \mathbb{R}$ is said to be a \emph{$\mathcal{K}$-function} if and only if it is continuous, strictly increasing and satisfies $\kappa(0)=0$.
   \end{definition}

This leads to the definition of persistent input trajectories.

\begin{definition}[Persistent input] \label{def:persistent_input} An input trajectory $u$ is said to be \emph{persistent} if and only if there exists $T>0$  such that,  for any $t\geq T$ there exists a $\mathcal{K}$-function, $\kappa_t$, such that for any $(\xi_1,\xi_2)\in (\mathbb{R}^{n_x})^2 $:
\begin{align}
   l(t-T,t,\xi_1,\xi_2,u)\geq\kappa_t(\Vert \xi_1-\xi_2\Vert).\label{eq:integral_persistent_input}
\end{align}
\end{definition}
Note that the definition of persistent input trajectories using a  $\mathcal{K}$-function differs from the one in \cite{besancon_nonlinear_2007} but they can be shown to be equivalent thanks to Lemma \ref{lem:kappa_cont_pos_def}.
 \begin{lemma}{[Lemma  4.3, \cite{khalil_nonlinear_2002}]}\label{lem:kappa_cont_pos_def}
 Let $n\in\mathbb{N}$ and $F:\mathbb{R}^n\rightarrow\mathbb{R}$ be a continuous function such that $F(0)=0$ and  for any $\xi \in \mathbb{R}^n\backslash\{0\}$, $F(\xi)>0$ then there exists a $\mathcal{K}$-function $\kappa$ such that for any $\xi \in \mathbb{R}^n$:
 \begin{align*}
     F(\xi)\geq \kappa(\Vert \xi \Vert).
 \end{align*}
 \end{lemma}

Persistent input trajectories allows one to distinguish every state during a time window of bounded length. In other words, one is then able to distinguish every pair of states without having to wait for more than a time span of length $T$. However, this property is not time-invariant. In some cases, $\kappa_t(\Vert \xi_1-\xi_2\Vert)$ might vanish as $t \rightarrow +\infty$ for fixed $\xi_1$ and $\xi_2$ making the system potentially less and less observable along the state trajectory. This can be avoided by considering a time invariant equivalent namely the regular persistence input trajectories.
\begin{definition}[Regularly persistent input] \label{def:regularly_persistent_input} An input trajectory $u$ is said to be \emph{regularly persistent} if and only if there exist $ T>0$ and  a  $\mathcal{K}$-function $\kappa$ such that for any $t\geq T$ and any  $ \xi_1\neq \xi_2 $:
\begin{align}
   l(t-T,t,\xi_1,\xi_2,u)\geq\kappa(\Vert \xi_1-\xi_2 \Vert).\label{eq:integral_regularly_persistent_input}
\end{align}

\end{definition}

\begin{remark}

 In Definitions \ref{def:cum_error_obs}, \ref{def:persistent_input} and \ref{def:regularly_persistent_input}, weighted Euclidian norms could be considered. This would lead to equivalent definitions because of the inequality relationships between weighted Euclidian norms and the standard one. Thus, it is without loss of generality that we limit our discussion to the standard Euclidian norm.
 \end{remark}

 It is very common to look for estimators that minimise the cumulative output error, see Chapter 4 of \cite{rawlings_model_2020} for a general review and analysis on the topic. As a consequence, the first contribution of this paper is to provide an interpretation of the previously stated integral observability  definitions in terms of optimization notions. This is the topic of next section.

%

\section{Observability and optimization-based estimation}\label{sec:obs_opti}

Optimization-based estimation aims to build estimators by minimizing a cost that depends on the input and output trajectories on some time interval. In this paper, we focus on this cost being the cumulative output error. One of the main theoretical issue in the deterministic setting is to ensure that the potential multiple solutions of the resulting optimization problems coincide locally or globally with the reference trajectory. In this section, we link the classical nonlinear observability concepts to Full Information and Moving Horizon Estimation.

To avoid confusion with the several definitions of observability stated above, we recall the definition of several concepts of solution of an optimization problem.

\begin{definition}
Let $F:\mathbb{R}^{n_x}\longrightarrow \mathbb{R}$. Consider the  optimization problem:
\begin{align}
\displaystyle \inf_{\xi\in \mathbb{R}^{n_x}} F(\xi).
\label{pb:gen_opt}
\end{align}
It is said that $\xi^* \in \mathbb{R}^{n_x}$ is a \emph{global solution} of Problem \eqref{pb:gen_opt} if for any $\xi \in \mathbb{R}^{n_x}$, $F(\xi^*)\leq F(\xi)$. It is said that $\xi^* \in \mathbb{R}^{n_x}$ is a \emph{local solution} of Problem \eqref{pb:gen_opt} if there exists a neighbourhood, $\mathbb{U}$, of $\xi^*$ such that for any $\xi \in \mathbb{U} $, $F(\xi^*)\leq F(\xi)$. It is said that $\xi^* \in \mathbb{R}^{n_x}$ is a \emph{strict local solution} of Problem \eqref{pb:gen_opt} if there exists a neighbourhood of $\xi^*$, $\mathbb{U}$, such that for any $\xi \in \mathbb{U}\backslash\{\xi^*\} $, $F(\xi^*) < F(\xi)$.
\end{definition}

\subsection{Nonlinear observability and optimisation }

We include straightforward properties of $l(t_1,t_2,\cdot,\cdot,u)$ for $0\leq t_1 \leq t_2$ and its derivatives in Lemma \ref{lem:cost_MHE_state} below. In the following, $ \diff_{\xi_2} l$ denotes the first order differential of $l(t_1,t_2,\xi_1,\cdot,u)$.
  \begin{lemma}
  \label{lem:cost_MHE_state} For any $\xi_1 \in \mathbb{R}^{n_x}$,  $0\leq t_1 \leq t_2$ and any input trajectory $u$,
  $l(t_1,t_2,\xi_1,\cdot,u)$ is continuously differentiable, $l(t_1,t_2,\xi_1,\xi_1,u)=0$, $\xi_1$ is a global solution of the following optimisation problem:
  \begin{equation}
\begin{array}{rrclcc}
\displaystyle\inf_{\xi_2\in \mathbb{R}^{n_x}}l(t_1,t_2,\xi_1,\xi_2,u),
\end{array}
\label{pb:horizon_state}
\end{equation}
 and ${\diff}_{\xi_2}l(t_1,t_2,\xi_1,\xi_1,u)=0$.
  \end{lemma}
  \begin{proof}
 Note that because $f$ and $h$ are continuously differentiable  then, according to Theorem 2.3.2 in \cite{bressan_introduction_2007}, for any $(\xi_1,\xi_2) \in (\mathbb{R}^{n_x})^2$, $0\leq t_1 \leq t_2$ and an input trajectory $u$, $l(t_1,t_2,\xi_1,\cdot,u)$ is  continuously differentiable too. Besides, $l(t_1,t_2,\xi_1,\xi_1,u)=0$ and $l(t_1,t_2,\xi_1,\xi_2,u) \geq 0$ from Definition \ref{def:cum_error_obs} which means that  $\xi_1$ is a global solution of Problem \eqref{pb:horizon_state}. As a consequence, from first order necessary optimality conditions of unconstrained problems, ${\diff}_{\xi_2}l(t_1,t_2,\xi_1,\xi_1,u)=0$.
  \end{proof}
  
\subsubsection{Full Information Estimation (FIE)}

\emph{Full Information} estimation is a straightforward optimization-based estimation technique. In FIE, the estimator is computed by minimising the cumulative measurement error  between the reference trajectory $x(\cdot)$ and an estimated trajectory $\phi_f(\cdot;0,\xi,u)$ on a interval $[0,t]$ for some $t\geq0$. 
It leads to the following optimization problem for any $t\geq0$ and $x_0 \in \mathbb{R}^{n_x}$:  
\begin{align}
\tag{$\text{FIE}_{t,u}$}
\displaystyle \inf_{\xi\in \mathbb{R}^{n_x}} l(0,t,x_0,\xi,u).
\label{pb:batch_gen_state}
\end{align}
Full Information estimation requires finding a global solution to Problem \eqref{pb:batch_gen_state}. Proposition \ref{prop:global_solution_universal_input} ensures that one recovers any initial condition $x_0$, if and only if $u$ is a universal input trajectory.
\begin{proposition}\label{prop:global_solution_universal_input}
For $t\geq0$, $u$ is a universal input trajectory on $[0,t]$ if and only if, for any $x_0\in \mathbb{R}^{n_x}$, $x_0$ is the unique global solution of Problem \eqref{pb:batch_gen_state}.
\end{proposition}
\begin{proof}
First, Lemma \ref{lem:cost_MHE_state} leads to  $l(0,t,\xi,x_0,u)\geq 0=l(0,t,x_0,x_0,u)$, for any $t\geq0$ and $(\xi,x_0) \in (\mathbb{R}^{n_x})^2$, so $x_0$ is a global solution of Problem \eqref{pb:batch_gen_state} independently of $u$. Then, by Proposition \ref{prop_universal_input}, $u$ is a universal input trajectory if and only if for any $\xi\neq x_0$, $l(0,t,\xi,x_0,u)>0$. This means that $x_0$ is the unique global solution of Problem \eqref{pb:batch_gen_state} for any $x_0 \in \mathbb{R}^{n_x}$  if and only if $u$ is a universal input trajectory.  
\end{proof}
Since the size of the integration window in \eqref{pb:batch_gen_state} grows with $t$, the numerical computation of  $l(0,t,x_0,\xi,u)$ and thus the practical resolution of  \eqref{pb:batch_gen_state} become progressively more difficult as time goes. A common alternative is to consider the input/output trajectories only on a time window of fixed length which leads to \emph{Moving Horizon} Estimation.

\subsubsection{Moving Horizon Estimation (MHE)}
As an alternative to Problem \eqref{pb:batch_gen_state}, one can consider a similar estimation problem where one keeps only the knowledge of $y(\cdot)$ on $[t-T,t]$ for some memory time $T>0$ and look for a \emph{Moving Horizon} estimator by minimising   $l(t-T,t,x(t-T),\xi,u)$ instead. This typically leads to the following optimization problem, for $t\geq T$:
\begin{equation}
\tag{$\text{MHE}_{t,T,u}$}
\begin{array}{rrclcc}
\displaystyle\inf_{\xi\in \mathbb{R}^{n_x}}l(t-T,t,x(t-T),\xi,u). 
\end{array}
\label{pb:receding_horizon_gen_state}
\end{equation}

Problem \eqref{pb:receding_horizon_gen_state} is written in the so-called `sequential form' where the goal is to recover $x(t-T)$ by solving Problem \eqref{pb:receding_horizon_gen_state} at time $t$ and reconstruct the rest of the trajectory by applying the flow $\phi_f$ with the input trajectory $u$. Similar to Problem \eqref{pb:batch_gen_state}, persistence of the input trajectory implies in particular uniqueness of a global solution of Problem \eqref{pb:receding_horizon_gen_state}.

\begin{proposition}
\label{prop:global_solution_pers_input}
An input trajectory $u$ is  persistent if and only if, there exists $T>0$ such that for any $t \geq T$ and  any initial condition $x_0 \in \mathbb{R}^{n_x}$, $x(t-T)=\phi_f(t-T;0,x_0,u)$ is the \emph{unique} global solution of Problem  \ref{pb:receding_horizon_gen_state}. 
\end{proposition}
\begin{proof}
The proof is very similar to that of Proposition \ref{prop:global_solution_universal_input}.


\end{proof}



\begin{remark}

Proposition \ref{prop:global_solution_pers_input} states that Moving Horizon Estimation is enabled by  persistent input trajectories. In the case of regularly persistent input trajectories, the presence of function $\kappa$ in Definition \ref{def:regularly_persistent_input} typically allows one to build global nonlinear observers. See Chapter 5 of \cite{besancon_nonlinear_2007} for an example. Related works have used similar conditions but they are considered uniformly with respect to control inputs. For example, in \cite{alessandri_moving-horizon_2008, rao_constrained_2003}, a condition called $N$-\emph{step observability} or uniform observability is assumed. It ensures that a small cumulative output error  on a rolling time window of size $N$ implies a small error in the initial conditions. This condition is formally very similar to the concept of uniform observability stated in Definition  \ref{def:universal_input} and ignores the influence that an input trajectory might have on observability.  Besides, in \cite{allan_robust_2021, hu_robust_2017,  knuefer_nonlinear_2021,knufer_time-discounted_2020,muller_nonlinear_2017}, global robust stabilty of FIE/MHE schemes are proved under a detectabilty assumption called \emph{incremental input/output-to-state stability} (i-IOSS) or its discounted version. Note that i-IOSS implies that the error between the current state of two trajectories can be bounded by the error in trajectories of process noise, measurement noise, control input and output. However, contrary to the observability conditions introduced in Section \ref{sec:weak_per}, the comparison functions used in i-IOSS  are independent of the control input which makes it a uniform detectability assumption. 
\end{remark}

\begin{remark}
Even if there exist regularly persistent input trajectories, they can be very hard to find because of the strong nature of the property. Moreover, one cannot hope to solve \eqref{pb:receding_horizon_gen_state} globally but only locally as it is generally nonconvex. Indeed, if one is only able to find local solutions of \eqref{pb:receding_horizon_gen_state}, then regular persistence seems unnecessary and  one needs a less demanding concept of observability. This notion of observability is discussed in the next section.
\end{remark}

 \section{Weak persistence and  Moving Horizon Estimation\label{sec:weak_per}  }
In this section, we introduce the notions of \emph{weakly and weakly regularly persistent} input trajectories that ensure  quantitative distinguishibility between  states that are near the reference one while having only access to the past observations on a moving time-window. These notions are  extensions of classical ones presented in Section \ref{sec:observability} and are designed to ensure that MHE problems can be solved. In particular, throughout this section, we show that weakly and weakly regularly persistent input trajectories ensure that the MHE has a locally unique local solution that is stable in the presence of small additive measurement noise.

\subsection{Definitions and first properties}
Note that a regularly persistent input trajectory $u$ is such that every possible state can be distinguished with the output of the system if one waits for no more than a fixed time. Thus, if one keeps the terminology from Definition \ref{def:local_weak_obs}, regular persistence is a strong and non-local property of the input trajectories. 
As mentioned previously, persistence of the input might be too demanding. Besides, it is generally very complicated to verify  that an input trajectory is persistent for a general nonlinear system because it requires checking that \eqref{eq:integral_persistent_input} holds for every pair of states. As a result, the concepts of persistent and regularly persistent inputs are too strong and unusable in many practical applications of MHE. One would prefer to ensure that only pairs of states in a neighborhood of  $x(t-T)$ are distinguishable on a rolling horizon for an appropriate choice of input trajectory.  
%
%
%
%
As a consequence, the second contribution of this paper is to state the definitions of weakly persistent and weakly regularly persistent input trajectories based on Definitions \ref{def:persistent_input} and \ref{def:regularly_persistent_input}.  They enable the practical resolution of  Moving Horizon Estimation problems and emphasize the role of the input trajectories in the proposed observability notions. Leveraging the notion of the Observability Grammian, we also give a necessary and a sufficient condition for weak and weak regular persistence of input trajectories based on second order derivatives. 

\begin{definition}[Weakly persistent input] \label{def:weakly_persistent_input} Fix an initial condition $x_0 \in \mathbb{R}^{n_x}$.  An input trajectory $u$ is said to be \emph{weakly persistent} at $x_0$, if there exists $T>0$ such that for any $ t\geq T$ there exist $R_t>0$ and a $\mathcal{K}$-function $\kappa_t$ such that for any $(\xi_1,\xi_2)\in (\widebar{B}(x(t-T),R_t))^2$:
\begin{align}
      l(t-T,t,\xi_1,\xi_2,u)\geq\kappa_t(\Vert \xi_1 -\xi_2 \Vert),\label{eq:weakly_per_input}
\end{align}
where $\widebar{B}(x(t-T),R_t)$ denotes the closed ball for the Euclidian norm centered at $x(t-T)$ of radius $R_t$.
System \eqref{eq:general_dyn_continous_time} is said to be weakly persistently observable if for any $x_0\in\mathbb{R}^{n_x}$ there exists a weakly persistent input trajectory at $x_0$.
\end{definition}

\begin{definition}[Weakly regularly persistent input] \label{def:weakly_regularly_persistent_input} Fix an initial condition $x_0 \in \mathbb{R}^{n_x}$.  An input trajectory $u$ is said to be \emph{weakly regularly persistent} at $x_0$, 
if there exists $T>0$, $R>0$ and a $\mathcal{K}$-function $\kappa$ such that, for any $ t\geq T$ and any $ (\xi_1,\xi_2)\in  (\widebar{B}(x(t-T),R))^2 $:
\begin{align}
      l(t-T,t,\xi_1,\xi_2,u)\geq\kappa(\Vert \xi_1-\xi_2 \Vert),\label{eq:weakly_reg_per_input}
\end{align}
where  $x(t-T)= \phi_f(t-T;0,x_0,u)$.
For $\mathcal{X}\subset \mathbb{R}^{n_x}$,  an input trajectory $u$ is said to be weakly regularly persistent on $\mathcal{X}$ if $u$ is  weakly regularly persistent at $x_0$ for any  $x_0\in\mathcal{X} $  and if $(T,R,\kappa)$ from  \eqref{eq:weakly_reg_per_input} depend only on $\mathcal{X}$ and $u$. System \eqref{eq:general_dyn_continous_time} is said to be weakly regularly observable if for any $x_0\in\mathbb{R}^{n_x}$ there exists a weakly regularly persistent input trajectory at $x_0$.
\end{definition}

It is clear from Definitions \ref{def:persistent_input} and \ref{def:regularly_persistent_input} that persistent input trajectories (resp. regularly persistent) are weakly persistent (resp. weakly regularly persistent).  Besides, from the properties of $\kappa$ in \eqref{eq:weakly_reg_per_input}, it is clear that weakly regularly persistent input trajectories are weakly persistent.
Roughly speaking, weakly persistently observable systems are such that, for some initial condition and some associated input trajectory,  the rolling cumulative measurement error between state trajectories starting close enough to the reference one does not vanish. Weakly regularly observable systems have the additional property that small rolling cumulative error in the output implies small `estimation' error uniformly in time. In the sequel, we give several characterizations of weakly and weakly regularly persistent input trajectories.

\subsection{Characterization of weakly persistent inputs} \label{sec:charac_weak_per_input}
Contrary to the observability concepts discussed in Section \ref{sec_class_nonlinear_obs},  weakly persistent input trajectories only ensure that Problem \eqref{pb:receding_horizon_gen_state} has a strict local and a global solution at $x(t-T)$ and potentially allows several global solutions. This is the topic of Proposition \ref{prop:local_solution_weak_per_input}.
\begin{proposition}\label{prop:local_solution_weak_per_input}
Let $x_0\in \mathbb{R}^{n_x}$ be an initial condition and $u$ be an input trajectory. Then, $u$ is a weakly persistent input trajectory at $x_0$ if and only if there exists $T>0$ such that for any $t\geq T$, there exists $R_t>0$ such that for any $\xi_1\in \widebar {B}(x(t-T),R_t)$,  the following optimisation problem:
\begin{equation}
\begin{array}{rrclcc}
\displaystyle\inf_{\xi_2\in \mathbb{R}^{n_x}}l(t-T,t,\xi_1,\xi_2,u)
\end{array}
\label{pb:receding_horizon_gen_state_neigh}
\end{equation}
admits a global solution at $\xi_1$ that is unique on $\widebar{B}(x(t-T),R_t)$.
In particular, in this case, $x(t-T)$ is a global solution and a strict local solution of Problem \eqref{pb:receding_horizon_gen_state}.
\end{proposition}

\begin{proof}
By definition, $u$ is a weakly persistent input trajectory at $x_0$ if and only if there exists $T>0$ such that for any $ t\geq T$ there exists $R_t>0$ and a a $\mathcal{K}$-function $\kappa_t$ such that for any $(\xi_1,\xi_2)\in (\widebar{B}(x(t-T),R_t))^2$ with $\xi_1\neq \xi_2$:
\begin{align}
      l(t-T,t,\xi_1,\xi_2,u)=\kappa_t(\Vert\xi_1-\xi_2\Vert)>0.\label{eq:weakly_per_input_proof}
\end{align}
From Lemma \eqref{lem:cost_MHE_state}, for any $x_0\in \mathbb{R}^{n_x}$, $T>0$. any $t\geq T$, any input trajectory $u$, and any  $(\xi_1,\xi_2)\in (\widebar{B}(x(t-T),R_t))^2$,  $l(t-T,t,\xi_1,\xi_1,u)=0$, $l(t-T,t,\xi_1,\xi_2,u)\geq 0$ and $\xi_1$ is a global solution of Problem \eqref{pb:receding_horizon_gen_state_neigh}. By also invoking Lemma \ref{lem:kappa_cont_pos_def}, this precisely means that $u$ is a weakly persistent input trajectory at $x_0$ if and only if there exists $T>0$ such that for any $ t\geq T$ there exists $R_t>0$ such that for any $\xi_1\in \widebar{B}(x(t-T),R_t)$, $\xi_1$ is the unique global solution of Problem \eqref{eq:weakly_per_input_proof} on $\widebar{B}(x(t-T),R_t)$. In the case that one of the two statements in Proposition \ref{prop:local_solution_weak_per_input} holds, one can choose $\xi_1=x(t-T)$ and \eqref{eq:weakly_per_input_proof} shows directly that $x(t-T)$ is a strict local solution of \eqref{pb:receding_horizon_gen_state}.  
\end{proof}

In the sequel, for $(n_1,n_2,m)\in \mathbb{N}^3$,  and any twice differentiable function $F:\mathbb{R}^{n_1}\rightarrow \mathbb{R}^m$  we  denote by $\diff F$ and $\diff^{2}F$ respectively the first and second order differential of $F$. In addition, for any twice differentiable function $G:\mathbb{R}^{n_1}\times \mathbb{R}^{n_2} \rightarrow\mathbb{R}^m$ and $(\xi_1,\xi_2)\in \mathbb{R}^{n_1}\times \mathbb{R}^{n_2}$, we denote by $\diff_{\xi_2} G(\xi_1,\xi_2)$ the differential of $G(\xi_1,\cdot)$ at $\xi_2$ and by $\diff^2_{\xi_2} G(\xi_1,\xi_2)$ the second order differential of $G(\xi_1,\cdot)$ at $\xi_2$.
In particular, for any $T>0$, any $t\geq T$, any input trajectory $u$ and any $(\xi_1,\xi_2)\in(\mathbb{R}^{n_x})^2$, we respectively  denote by $\diff_{\xi_2}l(t-T,t,\xi_1,\xi_2,u)$ and  $\diff^2_{\xi_2}l(t-T,t,\xi_1,\xi_2,u)$ the differential and the hessian  of  $l(t-T,t,\xi_1,\cdot,u)$ at $\xi_2$.  Their explicit expression are included in Lemma \ref{lem:1_2_diff_l} in Appendix \ref{app:1_2_diff_l} as well as a proof of their existence.

We first give the definition of a $\mathcal{K}$-function with finite sensitivity taken from \cite{alessandri_advances_2010}.

\begin{definition}[Finite sensitivity]\label{def:finite_sens}
A $\mathcal{K}$-function $\kappa$ is said to have \emph{finite sensitivity} if and only if there exists $r>0$ such that:
\begin{align}
    \inf_{\Vert \xi \Vert\neq0 ,\Vert \xi \Vert \leq r} \frac{\kappa(\Vert \xi \Vert)}{\Vert \xi \Vert^2}>0.\label{eq:finite_sensitivity}
\end{align}
\end{definition}
Intuitively, a  $\mathcal{K}$-function with finite sensitivity is  lower bounded by a positive definite quadratic form locally around $0$. As it is discussed in Proposition \ref{prop:hessian_pos_definite_weak_per}, this property allows one to link weak regular persistence of an input trajectory $u$ to the positive definiteness of the Hessian of $l(t-T,t,\xi_1,\cdot,u)$  at $\xi_2$, for any $(\xi_1,\xi_2)$ close to $x(t-T)$.
\begin{proposition}\label{prop:hessian_pos_definite_weak_per}
Let $x_0\in \mathbb{R}^{n_x}$ and $u$ be an input trajectory. Assume there exists $T>0$ such that for any $t\geq T$, there exists $R_t>0$ such that for any $(\xi_1,\xi_2)\in (\widebar{B}(x(t-T),R_t))^2$:
\begin{align}
    \mathrm{d}^2_{\xi_2}l(t-T,t,\xi_1,\xi_2,u)\succ 0,\label{eq:hess_pos_prop}
\end{align}
where $\succeq$ and $\succ$ denote Loewner partial order on positive semi-definite matrices. Then, $u$ is a weakly persistent input trajectory at $x_0$.

Conversely, if $u$ is a weakly persistent input trajectory at $x_0$ and all the associated $\mathcal{K}$-functions $\kappa_t$ have finite sensitivity, then there exists $T>0$ such that for any $t\geq T$, there exists $R_t>0$ such that for any $(\xi_1,\xi_2)\in (\widebar{B}(x(t-T),R_t))^2$, \eqref{eq:hess_pos_prop} holds.
\end{proposition}
\begin{proof}
See Appendix \ref{app:hessian_pos_definite_weak_per}.
\end{proof}

One of the main advantage of the concept of weak persistence is that it can be checked by computing the Observability Grammian of system \eqref{eq:general_dyn_continous_time} on a time interval of constant length. Its definition is stated in Definition \ref{def:grammian}.


\begin{definition}[Observability Grammian]\label{def:grammian}
Let $T>0$ be a time horizon, $x_0 \in \mathbb{R}^{n_x}$ be an initial condition and $u$ be an input trajectory. For $t\geq T$, the \emph{Observability Grammian}   of system \eqref{eq:general_dyn_continous_time} on $[t-T,t]$, denoted by $\mathcal{C}(t,T,x(t-T),u)$ is defined as half the Hessian of $l(t-T,t,x(t-T),\cdot,u)$ taken at $x(t-T)$ and reads:
\begin{align}
   \mathcal{C}(t,T,x(t-T),u)=&\frac{1}{2}\diff^2_{\xi_2}l(t-T,t,x(t-T),x(t-T),u),\notag\\
                                 =&\int_{t-T}^{t}\Phi_f^T H^T(x(s),u(s)) H(x(s),u(s))\Phi_f \mathrm{d}s, \label{eq:grammian_obs_gen}
\end{align}
where $H(x(s),u(s))=\diff_x h(x(s),u(s))$ and $\Phi_f(s;t-T,x(t-T),u)={\diff}_x\phi_f(s;t-T,\cdot,u)$.
\end{definition}

Lemma \ref{lem:obs_gram_hess} states the link between the Observability Grammian and the hessian of $l$ around $x(t-T)$.

\begin{lemma}\label{lem:obs_gram_hess}

There exists $T>0$ such that for any $t\geq T$,
\begin{align}
     \mathcal{C}(t,T,x(t-T),u)\succ 0,\label{eq:gram_pos_prop}
\end{align}
if and only if there exists $T>0$ such that for any $t\geq T$, there exists $R_t>0$ such that for any $(\xi_1,\xi_2)\in (\widebar{B}(x(t-T),R_t))^2$:
\begin{align}
    \mathrm{d}^2_{\xi_2}l(t-T,t,\xi_1,\xi_2,u)\succ 0.\label{eq:hess_pos_lem}
\end{align}

\end{lemma}
\begin{proof}
We first recall  that $\mathcal{C}(t,T,x(t-T),u)=\frac{1}{2}\mathrm{d}^2_{\xi_2}l(t-T,t,x(t-T),x(t-T),u)$. Thus by invoking the same continuity argument as in the proof of Proposition \ref{prop:hessian_pos_definite_weak_per} in Appendix \ref{app:hessian_pos_definite_weak_per}, one can show that if there exists $T>0$ such that for any $t\geq T$ \eqref{eq:gram_pos_prop} holds, then  there exists $T>0$ such that for any $t\geq T$, there exists $R_t>0$ such that for any $(\xi_1,\xi_2)\in (\widebar{B}(x(t-T),R_t))^2$:
\begin{align*}
    \mathrm{d}^2_{\xi_2}l(t-T,t,\xi_1,\xi_2,u)\succ 0.
\end{align*}
The converse follows by setting $(\xi_1,\xi_2)=(x(t-T),x(t-T))$ in \eqref{eq:hess_pos_lem}.
\end{proof}

Finally, Corollary \ref{cor:grammian_pos_definite_weak_per} gives another characterization of weakly persistent input trajectories in terms of positive definiteness of the Observability Grammian that is inspired by \cite{powel_empirical_2015}.

\begin{corollary}\label{cor:grammian_pos_definite_weak_per}
Let $x_0\in \mathbb{R}^{n_x}$ and $u$ be an input trajectory. If there exists $T>0$ such that for any $t\geq T$:
\begin{align}
     \mathcal{C}(t,T,x(t-T),u)\succ 0,\label{eq:gram_pos_prop_2}
\end{align}
then $u$ is a weakly persistent input trajectory at $x_0$.

Conversely, if $u$ is a weakly persistent input trajectory and all the associated $\mathcal{K}$-functions $\kappa_t$ have finite sensitivity then there exists $T>0$ such that for any $t\geq T$, \eqref{eq:gram_pos_prop_2} holds.
\end{corollary}
\begin{proof}
The result follows from Lemma \ref{lem:obs_gram_hess} and Proposition \ref{prop:hessian_pos_definite_weak_per}.
\end{proof}
\subsection{Characterization of weakly regularly persistent inputs}
In this section, we derive results in Proposition \ref{prop:hessian_lower_bound_definite_weak_reg_per} and \ref{prop:converse_hessian_lower_bound_weak_reg_per} that are the counterpart of those of Section \ref{sec:charac_weak_per_input} in the case of a weakly \emph{regularly} persistent input trajectory. The main conceptual difference between Proposition \ref{prop:hessian_pos_definite_weak_per} and Proposition \ref{prop:hessian_lower_bound_definite_weak_reg_per} is that one now requires the Hessian of $l(t-T,t,\xi_1,\cdot,u)$  at $\xi_2$ to be lower bounded independently of $t$ for $(\xi_1,\xi_2)$ in a neighbourhood of $x(t-T)$ whose radius is also independent of $t$.
\begin{proposition}\label{prop:hessian_lower_bound_definite_weak_reg_per}
Let $x_0\in \mathbb{R}^{n_x}$ be an initial condition and $u$ be an input trajectory.  If there exist $T>0$, $\mu>0$ and $R>0$ such that for any $t\geq T$ and for any $(\xi_1,\xi_2)\in (\widebar{B}(x(t-T),R))^2$:
\begin{align}
    \mathrm{d}^2_{\xi_2}l(t-T,t,\xi_1,\xi_2,u)\succeq \mu I_{n_x},\label{eq:hess_low_bound_prop}
\end{align}
where $I_{n_x}$ denotes the identity matrix of $\mathbb{R}^{n_x\times n_x}$,  then $u$ is a weakly regularly persistent input trajectory at $x_0$.

\end{proposition}

\begin{proof}

The result follows from the mean value form of the Taylor expansion of $l(t-T,t,\xi_1,\cdot,u)$ and is similar to that of Proposition \ref{prop:hessian_pos_definite_weak_per} in Appendix \ref{app:hessian_pos_definite_weak_per} with $\kappa(r)=\frac{\mu}{2}r^2$.
\end{proof}

A converse of Proposition \ref{prop:hessian_lower_bound_definite_weak_reg_per} in the spirit of  the second statement of Proposition \ref{prop:hessian_pos_definite_weak_per} is not straightforward. Indeed, the proof of the latter uses a continuity argument of  $\mathrm{d}^2_{\xi_2}l(t-T,t,\cdot,\cdot,u)$ at $x(t-T)$ to prove the existence of an adequate radius $R_t$. Because of the explicit dependence of $\mathrm{d}^2_{\xi_2} l$ on $t$, this argument does not allow one  to obtain a radius $R$ that is independent of $t$. Thus, new assumptions are needed to bridge the gap.

\begin{hypothesis}\label{as:three_times_diff}
The functions $f$ and $h$ are three times continuously differentiable.
\end{hypothesis}
\begin{hypothesis}\label{as:U_compact}
The set $U$ of feasible inputs is compact.
\end{hypothesis}

\begin{definition}\label{def:regular_bound}
Let $x_0 \in \mathbb{R}^{n_x}$ be an initial condition, $T>0$ a time horizon and  $u$ be an input trajectory. System \eqref{eq:general_dyn_continous_time} is said to be \emph{regularly bounded} at $x_0$ with horizon $T$ if there exist $R>0$ and $L>0$ such that for any $t\geq T$, any $s\in [t-T,t]$ and any $\xi\in \widebar{B}(x(t-T),R)$, 
\begin{align}
    \Vert  \phi_f(s;t-T,\xi,u) \Vert \leq L,\label{eq:regular_bound}
\end{align}
where  $x(t-T)= \phi_f(t-T;0,x_0,u)$.
\end{definition}

\begin{lemma}\label{lem:regular_bound_diff}
Let $x_0 \in \mathbb{R}^{n_x}$ be an initial condition, $T>0$ a time horizon and  $u$ be an input trajectory. Under Hypothesis \ref{as:U_compact}, if System \eqref{eq:general_dyn_continous_time} is regularly bounded at $x_0$ with horizon $T$ then there exist $L_1>0$, $L_2>0$ and $R>0$ such that for any $t\geq T$, any $s\in [t-T,t]$ and any $\xi\in \widebar{B}(x(t-T),R)$, 
\begin{align}
    \Vert  \Phi_f(s;t-T,\xi,u) \Vert &\leq L_1,\label{eq:regular_bound_1_diff}\\
      \Vert  \diff_{\xi} \Phi_f(s;t-T,\xi,u)\Vert &\leq L_2.\label{eq:regular_bound_2_diff}
      \intertext{Moreover, under Hypothesis \ref{as:three_times_diff}, there exist $L_3>0$ and $R>0$ such that for any $t\geq T$, any $s\in [t-T,t]$ and any $\xi\in \widebar{B}(x(t-T),R)$}
       \Vert  \diff^2_{\xi_2} \Phi_f(s;t-T,\xi,u)\Vert &\leq L_3,\label{eq:regular_bound_3_diff}
\end{align}
where $ \Vert\cdot\Vert$ denotes here the appropriate operator norm derived from the Euclidian norm.
\end{lemma}
\begin{proof}
See Appendix \ref{app:regular_bound_diff}.
\end{proof}

\begin{lemma}\label{lem:obs_gram_hess_regular}
Let $x_0 \in \mathbb{R}^{n_x}$ be an initial condition, $T>0$ a time horizon and  $u$ be an input trajectory. Under Hypothesis  \ref{as:three_times_diff} and \ref{as:U_compact}, if System \eqref{eq:general_dyn_continous_time} is regularly bounded at $x_0$ with horizon $T$ then the following statements are equivalent:

\begin{enumerate}[label=(\alph*)]
\item \label{item:lower_bound_grammian_obs_lem} There exists $\mu>0$ such that for any $t\geq T$:
\begin{align*}
     \mathcal{C}(t,T,x(t-T),u)\succeq \mu I_{n_x},
\end{align*}
where $\mathcal{C}$ is defined in Definition \ref{def:grammian};
\item \label{item:lower_bound_hess_lem}There exist $R>0$, $\mu>0$ such that for any $t\geq T$ and any $(\xi_1,\xi_2)\in (\widebar{B}(x(t-T),R))^2$:
\begin{align}
    \mathrm{d}^2_{\xi_2}l(t-T,t,\xi_1,\xi_2,u)\succeq \mu I_{n_x}.\label{eq:hessian_lower_bound_lem}
\end{align}
\end{enumerate}
\end{lemma}
\begin{proof}
See Appendix \ref{app:obs_gram_hess_regular}.
\end{proof}

\begin{proposition}\label{prop:converse_hessian_lower_bound_weak_reg_per}
Let $x_0 \in \mathbb{R}^{n_x}$ be an initial condition and $u$ be an input trajectory.
Under Hypotheses \ref{as:three_times_diff} and \ref{as:U_compact},  suppose that $u$ is a weakly regularly persistent input trajectory at $x_0$ with an associated $\mathcal{K}$-function $\kappa$ that has finite sensitivity and an associated time horizon $T$ such that System \eqref{eq:general_dyn_continous_time} is regularly bounded at $x_0$ with horizon $T$.  Then, there exist $T>0$, $R>0$, $\mu>0$ such that for any $t\geq T$ and any $(\xi_1,\xi_2)\in (\widebar{B}(x(t-T),R))^2$:
\begin{align}
    \mathrm{d}^2_{\xi_2}l(t-T,t,\xi_1,\xi_2,u)\succeq \mu I_{n_x}.\label{eq:hess_low_bound_prop_2}
\end{align}
\end{proposition}
\begin{proof}
See Appendix \ref{app:converse_hessian_lower_bound_weak_reg_per}.
\end{proof}

In the spirit of Corollary \ref{cor:grammian_pos_definite_weak_per}, Corollary \ref{cor:grammian_bound_definite_weak_reg_per} gives a  sufficient and a necessary condition for weak regular persistence in terms of lower boundedness  of the Observability Grammian uniformly with time.

\begin{corollary}\label{cor:grammian_bound_definite_weak_reg_per}
Let $x_0\in \mathbb{R}^{n_x}$ be an initial condition and $u$ be an input trajectory. Under Hypotheses \ref{as:three_times_diff} and \ref{as:U_compact}, if there exist $T>0$ and $\mu>0$ such that that System \eqref{eq:general_dyn_continous_time} is regularly bounded at $x_0$ with horizon $T$ and such that for any $t\geq T$:
\begin{align}
     \mathcal{C}(t,T,x(t-T),u)\succeq \mu I_{n_x},\label{eq:gram_bound_prop}
\end{align}
then $u$ is a weakly regularly persistent input trajectory at $x_0$.

Conversely, under Hypotheses \ref{as:three_times_diff} and \ref{as:U_compact},  suppose that $u$ is a weakly regularly persistent input trajectory at $x_0$ with an associated $\mathcal{K}$-function $\kappa$ that has finite sensitivity and an associated time horizon $T$ such that System \eqref{eq:general_dyn_continous_time} is regularly bounded at $x_0$ with horizon $T$, then  there exist $T>0$ and $\mu>0$ such that  for any $t\geq T$, \eqref{eq:gram_bound_prop} holds.
\end{corollary}
\begin{proof}



The result follows from Proposition \ref{prop:hessian_lower_bound_definite_weak_reg_per}, Lemma \ref{lem:obs_gram_hess_regular},  and Proposition \ref{prop:converse_hessian_lower_bound_weak_reg_per}.
\end{proof}

\begin{remark}\label{rem:MHE_estimation_error}
Although several notions of weakly persistent observability  have already been defined for MHE notably in \cite{alessandri_moving-horizon_2008,morari_efficient_2009,kang_moving_2006,wynn_convergence_2014},  weak persistence and weak regular persistence of the input do not seem to have been stated in this form and put into perspective with other nonlinear observability and optimization concepts. Furthermore, in the above cited work, it is typically assumed that a solution of a perturbed MHE problem is available. Then, it is  shown that, under an observability assumption very close to the one introduced in this paper, the estimation error of an approximate MHE scheme is ultimately bounded by the noise provided that it is small. However,  to the best of our knowledge, the stability of local solutions of MHE problems in the presence of output noise has not been treated so far in the literature.
\end{remark}

Following Remark \ref{rem:MHE_estimation_error}, it seems critical to wonder whether  small disturbances in the measurements imply a small drift between the true state of the system and the solution of a perturbed MHE problem under weak persistent or weak regular persistent observability. This the topic of Section \ref{sec:stability}.

\subsection{Stability of solutions of MHE problems under additive perturbation}\label{sec:stability}

In this section, we prove that weak and weak regular persistence of an input trajectory imply the existence and the local uniqueness of a local solution of the associated MHE problem in the presence of small additive process and output noise. A bound on the magnitude of the difference between the true state of the system and the perturbed local solution is also derived. In the case of a weakly regularly persistent input, this bound is independent of the time $t$.

In the following, for any $n\in \mathbb{N}$,  any $0\leq t_1<t_2\leq +\infty$, and any measurable $b:[t_1,t_2]\rightarrow \mathbb{R}^n$, we denote by $\Vert b \Vert_{\infty,[t_1,t_2]}$ the $L_\infty$ norm of $b$ on $[t_1,t_2]$. If $t_1=0$ and $t_2=+\infty$, we denote by $\Vert b \Vert_{\infty}$ the $L_\infty$ norm of $b$ on $\mathbb{R}^+$. We denote by $L_\infty([t_1,t_2],\mathbb{R}^n)$ the Banach space of measurable functions $b:[t_1,t_2]\rightarrow \mathbb{R}^n$ such that $\Vert b \Vert_{\infty,[t_1,t_2]}<+\infty$.  For any $T>0$  and $t\geq T$, we also set $\Theta_{t,T}= L_\infty([t-T,t],\mathbb{R}^{n_y})\times L_\infty([0,t],\mathbb{R}^{n_x})$ and $\Theta=L_\infty([0,+\infty[,\mathbb{R}^{n_y})\times L_\infty([0,+\infty[,\mathbb{R}^{n_x})$. For any $\eta=(v,w)\in \Theta_{t,T}$ (resp. $\Theta$), we set $\Vert \eta \Vert_{t,T}=\max(\Vert v \Vert_{\infty,[t-T,t]},\Vert w \Vert_{\infty,[0,t]})$ (resp. $\Vert \eta \Vert=\max(\Vert v \Vert_{\infty},\Vert w \Vert_{\infty})$). 
For any $T>0$ and any $t\geq T$, we respectively denote by $B_{t,\infty}$ and $\widebar{B}_{t,\infty}$ the open and closed ball in $\Theta_{t,T}$. We also use  ${B}_{\infty}$ and $\widebar{B}_{\infty}$  to denote the open and closed ball in  $\Theta$ respectively. We can now state the stability result of the section in both the cases of a weakly persistent and a weakly regularly persistent input trajectory.

First, for any $0\leq s_1\leq s \leq s_2 < +\infty $, any $\xi\in \mathbb{R}^{n_x}$, any control input trajectory $u$ and any process noise signal $w\in  L_\infty([s_1,s_2],\mathbb{R}^{n_x})$, we define the following perturbed Cauchy problem:
     \begin{align}
       \dot{x}(s)&=f(x(s),u(s))+w(s), \label{eq:general_perturbed_dyn_continous_time}\\
        x(s_1)&=\xi.\notag
   \end{align}
We assume for any  $0\leq s_1 \leq s\leq s_2< +\infty$, any $\xi\in \mathbb{R}^{n_x}$, any control input trajectory $u$ and any $w\in  L_\infty([s_1,s_2],\mathbb{R}^{n_x})$ that the solution of \eqref{eq:general_perturbed_dyn_continous_time} at time $s$, is uniquely defined and we denote it by $\tilde{\phi}_f(s;s_1,\xi,u,w)$. Since $w$ is only measurable, \eqref{eq:general_perturbed_dyn_continous_time} is only satisfied almost everywhere, see Theorem 2.1.1 and 2.1.3 in \cite{bressan_introduction_2007}.  Moreover, let  $x_0\in\mathbb{R}^{n_x}$ be an initial condition, $u$ be an input trajectory and  $w\in  L_\infty([0,t[,\mathbb{R}^{n_x})$ be a process noise signal. The perturbed reference trajectory  is defined  for any $t\geq0 $ by:
\begin{align}
    \tilde{x}(t,w)=\tilde{\phi}_f(t;0,x_0,u,w) \label{eq:reference_traj_perturbed}.
\end{align}
Note that \eqref{eq:reference_traj_perturbed} is a fortiori also defined for any $w\in  L_\infty([0,+\infty[,\mathbb{R}^{n_x})$. For conciseness, the dependence of $\tilde{x}$ on $u$ and $x_0$ is removed.  Clearly, $\tilde{\phi}_f$ coincides with ${\phi}_f$ in the unperturbed case leading to:
\begin{align*}
    \tilde{\phi}_f(s;s_1,\xi,u,0)&={\phi}_f(s;s_1,\xi,u),\\
    \tilde{x}(t,0)&={x}(t),
\end{align*}
where ${\phi}_f(s;s_1,\xi,u)$ and $x(t)$ are defined in \eqref{eq:general_dyn_continous_time} and \eqref{eq:reference_traj}.

Thus, for any $T>0$, any $t\geq T$, any $\xi\in \mathbb{R}^{n_x}$ and any perturbation signals $\eta=(v,w)\in \Theta_{t,T}$,  we define the perturbed version of \eqref{pb:receding_horizon_gen_state} as follows:
\begin{equation}
\tag{$\text{PMHE}_{t,T,u,v,w}$}
\begin{array}{rrclcc}
\displaystyle\inf_{\xi\in \mathbb{R}^{n_x}}\tilde{l}(t-T,t,\xi,u,\eta), 
\end{array}
\label{pb:receding_perturbed_horizon_gen_state}
\end{equation}
where for $\xi \in \mathbb{R}^{n_x}$:
\begin{align}
    \tilde{l}(t-T,t,\xi,u,\eta)=
        \int_{t-T}^{t}\Vert h(\tilde{x}(s,w),u(s))+v(s)- h(\phi_f(s;t-T,\xi,u),u(s))\Vert^2\mathrm{d}s.\label{eq:perturbed_error}
\end{align}


\begin{remark}
    Note that in \eqref{pb:receding_perturbed_horizon_gen_state}, $w$ and $v$ do not play the same role. Indeed, $v$ represents a \emph{measurement} noise that is pointwise additive in time, thus, only its values on the interval $[t-T,t]$ matter in the computation of $\tilde{l}$. On the contrary, $w$ is a \emph{process} noise that is integrated through \eqref{eq:general_perturbed_dyn_continous_time} Hence, the perturbed reference trajectory $\tilde{x}(\cdot,w)$ on $[t-T,t]$, depends on the values of $w$ on the whole interval $[0,t]$ and not only on those on $[t-T,t]$.   
\end{remark}

Consequently, in the following, we study the properties of $\tilde{x}$  and its differential when $w\in  L_\infty([0,t],\mathbb{R}^{n_x})$.    Lemma \ref{lem:bounded_x_bar} states the boundedness of $\tilde{x}(\cdot,w)$ on $[t-T,t]$  for any $t$  while Lemma \ref{lem:diff_process_noise} gives differentiability properties of $\tilde{x}(s,w)$ with respect to $w$ as well as the boundedness of the differential.

\begin{lemma}\label{lem:bounded_x_bar} 

Under Hypothesis \ref{as:U_compact},  for any $x_0 \in \mathbb{R}^{n_x}$, any $T>0$, any $t\geq T $, any input trajectory $u$, any process noise signal $w \in \textrm{L}_{\infty}([0,t],\mathbb{R}^{n_x}) $ and any $\nu_t >0$:
\begin{align}
      \sup_{s\in [t-T,t]} \sup_{\Vert w\Vert_{\infty,[0,t]}\leq \nu_t } \Vert \tilde{x}(s,w) \Vert <+\infty.
\end{align}

\end{lemma}
\begin{proof}
    See Theorem 3.2.3 in \cite{bressan_introduction_2007}.
\end{proof}

\begin{lemma}\label{lem:diff_process_noise}

Under Hypothesis \ref{as:U_compact},  for any $T>0$, $t\geq T$, $x_0 \in \mathbb{R}^{n_x}$, and $s \in [t-T,t]$, $\tilde{x}(s,\cdot)$ is continuously differentiable in $\textrm{L}_{\infty}([0,t],\mathbb{R}^{n_x})$. Furthermore, its differential is denoted by $\diff_{w}\tilde{x}(s,w)$ and  $\diff_{w}\tilde{x}(s,w)\Delta w=z(s)$  for any $(w,\Delta w) \in (\textrm{L}_{\infty}([0,t],\mathbb{R}^{n_x}))^2$, where $z$ is the unique solution of the following Cauchy problem for almost all $ s \in [0, t] $:
  \begin{align}
      \dot{z}(s)&= A(s)z(s) + \Delta w(s),\label{eq:diff_x_perturbed_w}\\
      z(0)&=0,\notag
  \end{align}
  with $A(s)=d_x f(\tilde{x}(s,w),u(s))$.

  
  Additionally,  $T>0$ $t\geq T$ and $\nu_t>0$: 
  \begin{align}
     \sup_{s\in [t-T,t]} \sup_{\Vert w\Vert_{\infty,[0,t]}\leq \nu_t } \Vert \diff_{w}\tilde{x}(s,w) \Vert <+\infty. \label{eq:bounded_x_perturbed_t}
  \end{align}

\end{lemma}
\begin{proof}
    See Theorem 3.2.6 in \cite{bressan_introduction_2007} and Proposition 5.1.1. in \cite{trelat_controoptimal_2005} for \eqref{eq:diff_x_perturbed_w} and Theorem 3.2.3 in \cite{bressan_introduction_2007} for \eqref{eq:bounded_x_perturbed_t}.
\end{proof} 
Note that the suprema in Lemma \ref{lem:bounded_x_bar} and \ref{lem:diff_process_noise} still depend on $t$. In order to get results that are uniform with respect to $t$, for an initial condition $x_0 \in \mathbb{R}^{n_x}$ and $u$ a control input trajectory, we introduce the following hypothesis:
\begin{hypothesis}\label{as:bounded_diff_x_tilde_w}

    There exists $\nu>0$ such that:
      \begin{align}
       \sup_{t\geq 0}  \sup_{\Vert w\Vert_{\infty}\leq \nu } \Vert \diff_w\tilde{x}(t,w) \Vert <+\infty.
  \end{align}

\end{hypothesis}

We can deduce the following Lemma:

\begin{lemma}\label{lem:boubnde_x_tilde_uniform}

   Let  $x_0 \in \mathbb{R}^{n_x}$ be an initial condition and  $u$ be a control input trajectory. Assume  Hypothesis \ref{as:bounded_diff_x_tilde_w} holds an that:
   \begin{align*}
        \sup_{t\geq 0}  \Vert {x}(t)\Vert <+\infty.
   \end{align*}
Then, there exists $\nu>0$ such that:
      \begin{align}
       \sup_{t\geq 0}  \sup_{\Vert w\Vert_{\infty}\leq \nu } \Vert \tilde{x}(t,w) \Vert <+\infty.
  \end{align}
  
\end{lemma}
\begin{proof}
    The result follows from the fact that, for any $t\geq 0$  any $\nu>0$, and any $w\in \widebar{B}_\infty(0,\nu)$, $ \Vert \tilde{x}(t,w) \Vert \leq  \Vert {x}(t) \Vert+ \Vert \tilde{x}(t,w) -{x}(t)\Vert$ and that, from the Mean Value Theorem:
    \begin{align*}
        \Vert \tilde{x}(t,w) -{x}(t)\Vert \leq  \nu \sup_{\Vert w\Vert_{\infty}\leq \nu } \Vert \diff_w\tilde{x}(t,w) \Vert. 
    \end{align*}
\end{proof}

Lemma \ref{lem:xi_v_diff_l_perturbed} in Appendix \ref{app:xi_v_diff_l_perturbed} gathers the important differentiability properties of $\tilde{l}$ with respect to $\xi$ and $\eta$ as well as the explicit expressions. We now state the first  result of the section knowing the  existence, local uniqueness and stability the solution of \eqref{pb:receding_perturbed_horizon_gen_state} in the case of weakly persistent input trajectory.

\begin{theorem}\label{th_weak_per_robustness}
Let $x_0 \in \mathbb{R}^{n_x}$ be an initial condition and $u$ be an input trajectory. Assume that $u$ is a weakly persistent input trajectory at $x_0$ and all the associated $\mathcal{K}$-functions $\kappa_t$ have finite sensitivity. Then, there exists $T>0$ such that for any $t\geq T$, there exist $\nu_t>0$, $R_t>0$, $K_t>0$ such that for any $\eta\in {B}_{t,\infty}(0,\nu_t)$, \eqref{pb:receding_perturbed_horizon_gen_state} has a unique local solution on $\widebar{B}(x(t-T),R_t)$ denoted by $\xi^*_t(\eta)$ and it satisfies:

\begin{align}
\Vert \xi^*_t(\eta)-x(t-T) \Vert\leq K_t   \Vert \eta \Vert_{t,T}.  \label{eq:error_solution_perturbed_weak_per}
\end{align}

\end{theorem}

\begin{proof}
See Appendix \ref{app:weak_per_robustness}.
\end{proof}

In order to prove the analogue of  Theorem \ref{th_weak_per_robustness} that involves time-independent quantities in the case of a weakly regularly persistent input trajectory,  we first prove a uniform Implicit Function Theorem on Banach spaces with explicit neighbourhoods. The classical Implicit Function Theorem on Banach spaces typically involves a pair $(x_0,y_0)$ valued in two Banach spaces and satisfying an equation of the form $F(x_0,y_0)=0$. The goal is then to prove the existence of $\delta>0$, $\epsilon>0$ and a function $\phi$ such that $y=\phi(x)$ if $\Vert x-x_0 \Vert< \delta $  and $\Vert y - y_0 \Vert <\epsilon$.  The idea of Proposition \ref{prop:uni_implicit_function_th} is to extend the classical Implicit Function Theorem to the case where one has a family of pairs of solutions $(x_{0,t},y_{0,t})_{t\in J}$ valued in Banach spaces and satisfying equations of the form  $F(t,x_{0,t},y_{0,t})=0$ for any $t \in J$. The main hurdle is that, in order to obtain a new interesting result in the  MHE analysis, one is looking for radii $\delta>0$ and $\epsilon>0$ that are uniform in $t$. The proof of Proposition \ref{prop:uni_implicit_function_th} is largely inspired by those of the Theorem in \cite{holtzman_explicit_1970} and Theorem 3.13 in \cite{pathak_introduction_2018}. 

\begin{proposition}[Uniform Implicit Function Theorem on Banach spaces with explicit neighbourhoods]\label{prop:uni_implicit_function_th}
Let $J$ be a set and $X$, $Y$, $Z$ be three Banach spaces. In the following, we do not distinguish the different norms, including those on linear operator spaces, and denote them by $\Vert \cdot \Vert$. Let $\Omega\subset  X\times Y$ be an open set and $F:J\times \Omega \rightarrow Z$ be a map on $J\times \Omega$. Let $\Omega^J$ be the set of mappings from $J$ to $\Omega$ and  let $((x_{0,t},y_{0,t}))_{t\in J}\in \Omega^J$  be a family of elements of $\Omega$ indexed by $J$. Let  $\epsilon>0$, $\delta>$, $L>0$ and $0<\alpha<1$ and for any $t\in J$, set $S_t=B(x_{0,t},\delta)\times \widebar{B}(y_{0,t},\epsilon)$ where $B$ and $\widebar{B}$ respectively denote the open and closed ball. Assume that:

\begin{enumerate}[label=(\roman*)]

\item \label{as:admissible_ball} for any $t \in J$, $S_t\subset \Omega$;
\item \label{as:equation_ref_proof} for any $t \in J$, $F(t,x_{0,t},y_{0,t})=0$;
\item \label{as:cont_diff_proof}for any $t \in J$, $F(t,\cdot,\cdot)$ is continuously differentiable on $\Omega$ so that, in particular, $\diff_yF(t,\cdot,\cdot)$ exists and is continuous on  $\Omega$;
\item \label{as:bounded_inverse_proof}for any $t \in J$, the linear operator $\diff_yF(t,x_{0,t},y_{0,t}):Y\rightarrow Z$ is invertible, and its inverse $\Gamma_t=\left(\diff_yF(t,x_{0,t},y_{0,t})\right)^{-1}$ is such that $\Vert\Gamma_t\Vert \leq L$;  
\item \label{as:lip_non_decreasing_proof} there exists $g_1:[0,\delta]\times[0,\epsilon]\rightarrow\mathbb{R}$ such that for any $r \in [0,\delta]$ and any $s \in [0,\epsilon]$,  $g_1(r,\cdot)$ and $g_1(\cdot,s)$ are non-decreasing and such that, for any $t\in J$ and any $(x,y)\in S_t$:
\begin{align*}
    \Vert \diff_yF(t,x,y)-\diff_yF(t,x_{0,t},y_{0,t})\Vert \leq g_1(\Vert x- x_{0,t}\Vert,\Vert y- y_{0,t}\Vert);
\end{align*}
\item \label{as:lip_non_decreasing_2_proof} there exists a non-decreasing function $g_2:[0,\delta]\rightarrow\mathbb{R}$ such that for any $t\in J$ and any $x\in B(x_{0,t},\delta)$:
\begin{align*}
    \Vert F(t,x,y_{0,t})\Vert\leq  g_2(\Vert x- x_{0,t}\Vert);
\end{align*}
\item \label{as:condition_delta_eps_proof}the positive numbers $\delta$, $\epsilon$, $L$ and $\alpha$ satisfy: \begin{align*}
    &Lg_1(\delta,\epsilon)\leq \alpha<1,
    &Lg_2(\delta)\leq\epsilon(1-\alpha). 
\end{align*}
\end{enumerate}

Then, for any $t\in J$, there exists a unique continuously differentiable maps $\phi_t: B(x_{0,t},\delta)\rightarrow B(y_{0,t},\epsilon) $ such that:

\begin{enumerate}[label=(\alph*)]
    \item $y_{0,t}=\phi_t(x_{0,t})$;\label{eq:phi_ref_proof}
    \item for any $x\in B(x_{0,t},\delta)$, $F(t,x,\phi_t(x))=0$;\label{eq:implicit_proof}
    \item \label{eq:diff_phi}for any $x\in B(x_{0,t},\delta)$, $\diff_yF(t,x_,\phi_t(x))$ is invertible and
    \begin{align*}
     \Vert(\diff_y&F(t,x,\phi_t(x)))^{-1}\Vert \leq \frac{L}{1-Lg_1(\delta,\epsilon)},\\
        \diff \phi_t(x)&=(\diff_yF(t,x_,\phi_t(x)))^{-1}\diff_xF(t,x_,\phi_t(x)),
    \end{align*}
    
    where $(\diff_yF(t,x_,\phi_t(x)))^{-1}:R(t,x)\subset Z\rightarrow X$ and $R(t,x)=image(\diff_yF(t,x,\phi_t(x)))$.
\end{enumerate}

\end{proposition}
\begin{proof}
See Appendix \ref{app:uni_implicit_function_th}.
\end{proof}
 We can now state the main result of the section.

\begin{theorem}\label{th_weak_reg_per_robustness}
Let $x_0 \in \mathbb{R}^{n_x}$ be an initial condition, $u$ be an input trajectory. Assume that  Hypotheses \ref{as:three_times_diff}, \ref{as:U_compact} and \ref{as:bounded_diff_x_tilde_w} hold and that $u$ is a weakly regularly persistent input trajectory at $x_0$ with an associated $\mathcal{K}$-function $\kappa$ that has finite sensitivity and an associated time horizon $T$ such that System \eqref{eq:general_dyn_continous_time} is regularly bounded at $x_0$ with horizon $T$. Then, there exist $\mu>0$, $R'>0$, $\nu'>0$ and variable-wise non-decreasing functions vanishing at $0$, $g_1:\mathbb{R}^+\times\mathbb{R}^+\rightarrow\mathbb{R}^+$, $g_2:\mathbb{R}^+\rightarrow\mathbb{R}^+$  $g_3:\mathbb{R}^+\times\mathbb{R}^+\rightarrow\mathbb{R}^+$  such that for any $0<\nu<\nu'$, $0<R<R'$, $0<\alpha<1$ and any $\eta\in {B}_{\infty}(0,\nu)$, if
\begin{align}
     &\frac{g_1(\nu,R)}{\mu}\leq \alpha<1, &\frac{g_2(\nu)}{\mu}\leq R (1-\alpha),\label{eq:conditions_neigh_th_weak_reg}
\end{align}
then for any $t\geq T$,  \eqref{pb:receding_perturbed_horizon_gen_state} has a unique local solution in $\widebar{B}(x(t-T),R)$, denoted by $\xi^*_t(\eta)$, and it satisfies:
\begin{align}
\Vert \xi^*_t(\eta)-x(t-T) \Vert\leq  \frac{g_3(\nu,R)}{\mu-g_1(\nu,R)}   \Vert\eta \Vert.  \label{eq:error_solution_perturbed_weak_reg_per}
\end{align}
\end{theorem}

\begin{proof}
See Appendix \ref{app:th_weak_reg_per_robustness}.


    
 
\end{proof}



\begin{remark}
Both results in Theorem \ref{th_weak_per_robustness} and Theorem \ref{th_weak_reg_per_robustness} state that the distance between the local solution of \eqref{pb:receding_perturbed_horizon_gen_state} and the true state is at most proportional to the norm of the measurement and process noises. Theorem \ref{th_weak_per_robustness} is a direct consequence of the classical Implicit Function Theorem. We state it as it gives an element of comparison to Theorem \ref{th_weak_reg_per_robustness} which is the main contribution of the section. Indeed, depending on the evolution  of $K_t$ in $t$ in Theorem \ref{th_weak_per_robustness}, the norm of the noise $\Vert \eta \Vert_{t,T}$ may need to vanish when $t$ goes to infinity in order to keep the right-hand side bounded. It significantly limits the class of noise trajectories that can be dealt with by the system. On the contrary, in the setting of Theorem \ref{th_weak_reg_per_robustness}, the stability of the solution of  \eqref{pb:receding_perturbed_horizon_gen_state} is ensured for any sufficiently small bounded perturbation trajectory since the parameters $(R,\nu,\mu)$ do not depend on $t$. Therefore, weak regular persistence is more useful in practice than weak persistence.   
\end{remark}

\begin{remark}
The explicit expressions of $g_1$ and $g_2$ in Theorem \ref{th_weak_reg_per_robustness} are not included in order to clarify its link to the time-uniform Implicit Function Theorem as presented in Proposition \ref{prop:uni_implicit_function_th}. However, by looking more closely at Equations \eqref{eq:lip_constant_xi_v_xi_xi_diff_4} and \eqref{eq:lip_constant_v_xi_diff_2} in Appendix \ref{app:th_weak_reg_per_robustness} then the conditions \eqref{eq:conditions_neigh_th_weak_reg} read:
\begin{align}
    &a_1(\nu,R)(\nu+R)\leq \alpha\mu,&\frac{a_2(\nu)\nu}{R}\leq (1-\alpha)\mu,\label{eq:conditions_detailed}
\end{align}
where $a_2(\nu)>0$ and $a_1(\nu,R)>0$ are non-decreasing with respect to $\nu$ and $R$. 
Note that $R$ represents the radius of the neighbourhood of $x(t-T)$ where \eqref{pb:receding_perturbed_horizon_gen_state} is known to have a unique local solution and that $\nu$ represents the maximal amount of noise allowed in order to keep stability. It is clear that \eqref{eq:conditions_detailed}  encodes a trade-off between $R$ and $\nu$ regulated by the choice of $0<\alpha<1$ which is arbitrary. 
\end{remark}

\begin{remark}
Similar existing observability concepts for MHE which can be found in \cite{alessandri_moving-horizon_2008,alessandri_moving-horizon_2016,morari_efficient_2009,kang_moving_2006,wynn_convergence_2014} are all stated uniformly with respect to the input unlike the ones introduced in this paper. In Section \ref{sec:example_loc}, we provide an example of system that is not observable uniformly with respect to the input.
\end{remark}

\section{An example: bearing-only localisation }\label{sec:example_loc}
In this section, we present the problem of bearing-only localisation where one wants to recover the position of a mobile sensor using measurements of the direction toward a beacon.  Thus, we consider the following  2D dynamics and observation equation:

\begin{align}
    \dot{x}=&u,\label{eq:dyn_1landmark_continuous_time}\\
    y=&h(x):=\frac{\ell-x}{\Vert \ell-x\Vert},
\end{align}
where $\ell\in \mathbb{R}^2$ is assumed to be known a priori and $u$ is an input trajectory valued in $\mathbb{R}^2$. Let $x_0 \in \mathbb{R}^2$ be an initial condition such that $x_0 \neq \ell$ and $t_0=0$ be the reference initial time. In this case the solution flow $\phi$ and its differential $\Phi$ read: 
\begin{align}
    \phi(t;0,x_0,u)=&x_0+\int_0^tu(s)\mathrm{d}s,\label{eq:flow_single_int}\\
    \Phi(t;0,x_0,u)=&I_2.\label{eq:diff_flow_single_int}
\end{align}
 For any $\xi=(\xi_1,\xi_2)\in \mathbb{R}^2$ and any $\ell=(\ell_1,\ell_2)\in \mathbb{R}^2$ such that $r=\Vert \ell-\xi\Vert>0$, let  $H=\diff  h$ be:
    \begin{align}
      H(\xi)=  \frac{1}{r^3}
    \begin{bmatrix}
    -(\xi_2-\ell_2)^2 & (\xi_1-\ell_1)(\xi_2-\ell_2)\\
    (\xi_1-\ell_1)(\xi_2-\ell_2) &  -(\xi_1-\ell_1)^2
    \end{bmatrix}.\label{eq:diff_h_bearing}
    \end{align}

From \eqref{eq:flow_single_int}, \eqref{eq:diff_flow_single_int}, and \eqref{eq:diff_h_bearing} and straightforward computations one gets, for any $T>0$, any $t\geq T$ and any $\xi \in \mathbb{R}^2$, that:
\begin{align*}
   \mathcal{C}(t,T,\xi,u)= \int_{t-T}^{t}  \frac{1}{{r}^4(s)}\begin{bmatrix}
                                   e_2^2(s) &           -e_1(s)e_2(s)\\
                                    -e_1(s)e_2(s) &  e_1^2(s)
                                \end{bmatrix}\mathrm{d}s,
\end{align*}
where ${r}(s)=\Vert \ell-{x}(s)\Vert$, $e(s)=(e_1(s),e_2(s))={x}(s)-\ell$ and ${x}(s)=\phi(s;0,x_0,u)$. In the following, we define three classes of input trajectories. 
 \begin{enumerate}[label=\arabic*.,wide, labelindent=0pt]
    \item \textit{Radial constant input trajectory}
    
     for any $\sigma \in \mathbb{R}$, and any $s\geq0$, we define the radial constant input trajectory $u_{cst}(s,\sigma)$ as follows:
     \begin{align}
         u_{cst}(s,\sigma)=\sigma(\ell-x_0);\label{eq:ex_input_cst}
     \end{align}
     \item  \textit{Circular input trajectory}
     
     for any $\omega>0$ and $r_0>0$ and any $s\geq0$, we define the circular input trajectory  as follows:
     \begin{align}
    u_{circ}(s,\omega,r_0)&=\omega r_0\begin{bmatrix}
                                        -\sin(\omega s+{\psi}_0)\\
                                        \cos(\omega s+{\psi}_0)
                                      \end{bmatrix},\label{eq:ex_input_circ}
     \end{align}
  
     where $r_0=\Vert \ell-x_0\Vert$, $\psi_0=\textrm{atan2}(\ell_1-x_{0,1},\ell_2-x_{0,2})$ and $x_0=(x_{0,1},x_{0,2})$;
     
     \item  \textit{Outward spiral input trajectory}
     
      for any $\omega>0$, $\alpha>0$, $r_0>0$ and any $s\geq0$, we define the outward spiral trajectory  as follows:
     \begin{align}
    u_{spi}(s,\omega,\alpha,r_0)&=\omega r_0\exp(\alpha s)\left(\begin{bmatrix}
                                        -\sin(\omega s+{\psi}_0)\\
                                        \cos(\omega s+{\psi}_0)
                                      \end{bmatrix}+\alpha\begin{bmatrix}
                                        \cos(\omega s+{\psi}_0)\\
                                        \sin(\omega s+{\psi}_0)
                                      \end{bmatrix}\right).\label{eq:ex_input_spi}
     \end{align}
    \end{enumerate}
 In Proposition \ref{prop:ex_per}, we show that input trajectories in \eqref{eq:ex_input_cst} represent those that are not weakly persistent. Then we show that input trajectories in \eqref{eq:ex_input_circ} represent weakly regularly persistent ones  and that \eqref{eq:ex_input_spi} represent weakly persistent input trajectories such that the associated Observability Grammian can never be lower bounded as in Corollary \ref{cor:grammian_bound_definite_weak_reg_per}.

\begin{proposition}\label{prop:ex_per}

The following statements are true:
\begin{itemize}
    \item For any $\sigma \in \mathbb{R}$, $u_{cst}(\cdot,\sigma)$ is neither  a universal input trajectory of system \eqref{eq:dyn_1landmark_continuous_time} nor a weakly persistent input trajectory  at $x_0$.
    \item For any $\omega>0$ and $r_0>0$ ,  $u_{circ}(\cdot,\omega,r_0)$ is a weakly regularly persistent input trajectory of system \eqref{eq:dyn_1landmark_continuous_time} at $x_0$.
    \item   For any $\omega>0$, $\alpha>0$ and $r_0>0$,  $u_{spi}(\cdot,\omega,\alpha,r_0)$ is a weakly persistent input trajectory of system \eqref{eq:dyn_1landmark_continuous_time} at $x_0$ and for any $T>0$, $\lim_{t\rightarrow +\infty}\Vert  \mathcal{C}(t,T,x(t-T),u_{spi})\Vert=0$.
\end{itemize}

\end{proposition}
\begin{proof}
See Appendix \ref{app:ex_per}.
\end{proof}

\begin{remark}
The first item in Proposition \ref{prop:ex_per} shows that System \eqref{eq:dyn_1landmark_continuous_time} is not uniformly observable in the sense of  Definition \ref{def:universal_input}. Thus, the MHE algorithms mentioned in Remark \ref{rem:MHE_estimation_error}, which require uniform observability properties, could not directly be applied to this example without an adequate choice of input trajectory. However, the second item in Proposition \ref{prop:ex_per} shows that by using circular input trajectories, one recovers the properties of the associated MHE problem discussed in Section \ref{sec:weak_per}. The last item shows that there exist input trajectories that can never be proved to be weakly regularly persistent using the Observability Grammian.

\end{remark}

\begin{remark}
   Note that System \eqref{eq:dyn_1landmark_continuous_time} does not satisfy Hypothesis \ref{as:bounded_diff_x_tilde_w} as any nonzero constant process noise would make $\tilde{x}$ unbounded. However, System \eqref{eq:dyn_1landmark_continuous_time} could be modified by adding a linear  locally stabilising control feedback term to $u_{circ}$ in order to robustly track a circle for example. In this case, Hypothesis \ref{as:bounded_diff_x_tilde_w} would hold as System     \eqref{eq:diff_x_perturbed_w} in Lemma \eqref{lem:diff_process_noise} would become robustly stable in the presence of small perturbations.
\end{remark}
\section*{Conclusion}

In this paper, we have first studied connections between classical nonlinear observability and optimisation notions. Then, we have introduced the concepts of weakly and weakly regularly persistent input trajectory along with their connection to the Observability Grammian and the existence and uniqueness of solutions to the problem of Moving Horizon Estimation.  Then, thanks to a specifically designed time-uniform Implicit Function Theorem, we have shown that these conditions imply the stability of MHE solutions with respect to small additive perturbations in the measurements both uniformly and non-uniformly in time. Finally, we presented an example of a nonlinear system where classical uniform observability conditions do not hold along with  examples and counter-examples of weakly persistent and  weakly regularly persistent input trajectories. In future works, one could introduce an arrival cost in the MHE problem and study the stability properties of the resulting optimisation problem in the spirit of \cite{zavala_stability_2010}.

\section*{Acknowledgements}
This work received funding from the Australian Government, via grant AUSMURIB000001 associated with ONR MURI grant N00014-19-1-2571. We would like to thank Abhishek Bhardwaj for his comments.
\appendix

\section{Differentials of $l$}\label{app:1_2_diff_l}

\begin{lemma}[First and second order differential of  $l(t-T,t,\xi_1,\cdot,u)$ ] \label{lem:1_2_diff_l}
For any $T>0$, any $t\geq T$, any input trajectory $u$, $l(t,T,\xi_1,\cdot,u)$ is twice continuously differentiable and for any $(\xi_1,\xi_2)\in(\mathbb{R}^{n_x})^2$, $\diff_{\xi_2}l(t-T,t,\xi_1,\xi_2,u)$ and  $\diff^2_{\xi_2}l(t-T,t,\xi_1,\xi_2,u)$ read:
\begin{align}
    &\diff_{\xi_2}l(t-T,t,\xi_1,\xi_2,u)=2\int_{t-T}^{t}(h(x_2(s),u(s))-h(x_1(s),u(s)))^T H(x_2(s),u(s))\Phi_f(s,\xi_2) \mathrm{d}s, \label{eq:1_diff_l}
\end{align}
where  $H(\xi_1,u(s))=\diff_x h(\xi_1,u(s))$, $x_1(s)=\phi_f(s;t-T,\xi_1,u)$, $x_2(s)=\phi_f(s;t-T,\xi_2,u)$ and  $\Phi_f(s,\xi_2)={\diff}_x\phi_f(s;t-T,\xi_2,u)$.
\begin{align}
    \diff^2_{\xi_2}l(t-T,t,\xi_1,\xi_2,u) = 2\mathcal{C}&(t,T,\xi_2,u)+2\mathcal{R}(t,T,\xi_1,\xi_2,u),\label{eq:2_diff_l}
\end{align}
where for any $(\Delta \xi_2,\Delta' \xi_2) \in (\mathbb{R}^{n_x})^2$: 
\begin{align*}
   \Delta \xi_2^T\mathcal{C}(t,T,\xi_2,u)\Delta' \xi_2 &=
   \int_{t-T}^{t}\Delta \xi_2^T\Phi_f(s,\xi_2)^T H^T(x_2(s),u(s)) H(x_2(s),u(s))\Phi_f(s,\xi_2)\Delta' \xi_2 \mathrm{d}s,\\
     \Delta \xi_2^T\mathcal{R}(t,T,\xi_1,\xi_2,u)\Delta' \xi_2&=
    \int_{t-T}^{t}(h(x_2(s),u(s))-h(x_1(s),u(s)))^T \mathrm{}{\xi_2}(H\Phi_f(s,\xi_2)\Delta' \xi_2)\cdot\Delta \xi_2  \mathrm{d}s,
\end{align*}
where for any $\xi_2 \in \mathbb{R}^{n_x}$, $H\Phi_f(s,\xi_2)=H(x_2(s),u(s))\Phi_f(s,\xi_2)$ and for any $\Delta \xi_2 \in \mathbb{R}^{n_x}$:
\begin{align*}
    &\diff_{\xi_2}H\Phi_f(s,\xi_2)\cdot \Delta \xi_2 =\\
    &H(x_2(s),u(s)) (\diff_{\xi_2}\Phi_f(s,\xi_2)\cdot\Delta \xi_2)+(\diff_{x}H(x_2(s),u(s))\cdot(\Phi_f(s,\xi_2) \cdot \Delta \xi_2))\Phi_f(s,\xi_2).
\end{align*}
Besides, $\diff^2_{\xi_2}l(t-T,t,\xi_1,\xi_1,u) = 2\mathcal{C}(t,T,\xi_1,u)$.

\end{lemma}
\begin{proof}
 Note that, for any $T>0$, any $t\geq T$, any input trajectory $u$, and $\xi \in \mathbb{R}^{n_x}$ and according to Theorem 2.3.2 in \cite{bressan_introduction_2007} applied twice, $\phi_f(s;t-T,\xi_1,u)$ is twice continuously differentiable since $f$ is. Since $h$ is also twice continuously differentiable then $\diff_{\xi_2}l$ and $\diff^2_{\xi_2}l$ exist. Note that \eqref{eq:1_diff_l} and \eqref{eq:2_diff_l} can be derived by the theorems of derivation inside integrals and the application of the chain rule while the last equation results from the fact that for any  $\xi_1 \in \mathbb{R}^{n_x}$, $\mathcal{R}(t,T,\xi_1,\xi_1,u)=0$.
\end{proof}

\section{Proof of Proposition \ref{prop:hessian_pos_definite_weak_per}}\label{app:hessian_pos_definite_weak_per}

\begin{proof}
 Assume that  there exists $T>0$ such that for any $t\geq T$, there exists $R_t>0$ such that for any $(\xi_1,\xi_2)\in (\widebar{B}(x(t-T),R_t))^2$,  \eqref{eq:hess_pos_prop} holds. From Lemma \ref{lem:cost_MHE_state}, for $T>0$ and $t\geq T$, and $\xi_1\in \mathbb{R}^{n_x}$ one has $l(t,T,\xi_1,\xi_1,u)=0$  and ${\diff}_{\xi_2}l(t,T,\xi_1,\xi_1,u)=0$. Moreover, from the mean value form of the Taylor expansion of $l(t-T,t,\xi_1,\cdot,u)$ at $\xi_1$ (see Equation $(b)$ in  Proposition A.23 of \cite{bertsekas_nonlinear_1997}), one has that for any $(\xi_1,\xi_2)\in (\widebar{B}(x(t-T),R_t))^2$: 
\begin{align}
    l&(t,T,\xi_1,\xi_2,u)=\frac{1}{2}(\xi_2-\xi_1)^T\mathrm{d}^2_{\xi_2}l(t-T,t,\xi_1,\chi,u)(\xi_2-\xi_1),\label{eq:mean_value_l_state}
\end{align}
with $\chi=(1-\lambda)\xi_1+\lambda \xi_2$ and $0<\lambda<1$. Since $\chi \in\widebar {B}(x(t-T),R_t) $, one has, from  \eqref{eq:hess_pos_prop}, that $\mathrm{d}^2_{\xi_2}l(t-T,t,\xi_1,\chi,u)\succ 0$. By denoting by $\mu_t$ the smallest eigenvalue of $\mathrm{d}^2_{\xi_2}l(t-T,t,\xi_1,\chi,u)$, one gets that for any $(\xi_1,\xi_2)\in (\widebar{B}(x(t-T),R_t))^2$, $l(t,T,\xi_1,\xi_2,u)\geq\frac{\mu_t}{2}\Vert\xi_2-\xi_1\Vert^2$,
and the results is proven by choosing $\kappa_t(r)=\frac{\mu_t}{2}r^2$.
For the converse, assume that if $u$ is a weakly persistent input trajectory at $x_0$ and all the associated $\mathcal{K}$-functions $\kappa_t$ have finite sensitivity. Then, there exists $T>0$ such that for any $ t\geq T$ there exist $R_t>0$, and a $\mathcal{K}$-function $\kappa_t$ such that for any $(\xi_1,\xi_2)\in (\widebar{B}(x(t-T),R'_t))^2$:
\begin{align}
      l(t-T,t,\xi_1,\xi_2,u)\geq\kappa_t(\Vert \xi_1 -\xi_2 \Vert),
\end{align}
\begin{align}
   \mu'_t= \inf_{\Vert \xi \Vert \leq R'_t} \frac{\kappa_t(\Vert \xi \Vert)}{\Vert \xi \Vert^2}>0.
\end{align}
In particular, for any $\xi\in\widebar {B}(x(t-T),R'_t)$:
\begin{align}
     l(t-T,t,x(t-T),\xi,u)\geq \mu'_t \Vert \xi- x(t-T) \Vert^2.\label{eq:lower_bound_l_proof}
\end{align}
                  
From the Taylor's expansion of $l(t-T,T,x(t-T),\cdot,u)$ at $x(t-T)$, see  Equation $(c)$ in  Proposition A.23 of \cite{bertsekas_nonlinear_1997},  and from Lemma \ref{lem:cost_MHE_state}, for $\xi$ in a neighborhood of  $x(t-T)$, one gets that:
\begin{align}
    &l(t,T,x(t-T),\xi,u)=\frac{1}{2}w^T\mathrm{d}^2_{\xi_2}l(t-T,t,x(t-T),x(t-T),u)w +\Vert w \Vert^2\theta(\xi),\label{eq:taylor_l}
\end{align}
where $w=\xi-x(t-T)$ and  $\lim_{\xi \rightarrow x(t-T)}\theta(\xi)=0$.
By further combining  \eqref{eq:lower_bound_l_proof} and \eqref{eq:taylor_l}, one gets for $\xi$ in a neighborhood of $x(t-T)$ such that $\xi\neq x(t-T)$:
\begin{align*}
    \frac{w^T \mathrm{d}^2_{\xi_2}l(t-T,t,x(t-T),x(t-T),u)w}{\Vert w \Vert^2}&\geq \tilde{\mu'_t}+2\theta(\xi),\\
    \frac{w^T \mathrm{d}^2_{\xi_2}l(t-T,t,x(t-T),x(t-T),u) w}{\Vert w \Vert^2}&\geq \mu_t,
\end{align*}
where $\mu_t =\frac{\tilde{\mu_t}}{2}$. Thus, $\mathrm{d}^2_{\xi_2}l(t-T,t,x(t-T),x(t-T),u)\succ 0$. Then by continuity of the smallest eigenvalue and of $(\xi_1,\xi_2)\rightarrow \mathrm{d}^2_{\xi_2}l(t-T,t,\xi_1,\xi_2,u)$, there exist $R_t>0$ such that for any $(\xi_1,\xi_2)\in (\widebar{B}(x(t-T),R_t))^2$, $ \mathrm{d}^2_{\xi_2}l(t-T,t,\xi_1,\xi_2,u)\succ0$, and the result is proven. 
\end{proof}


\section{Proof of Lemma \ref{lem:regular_bound_diff}}\label{app:regular_bound_diff}

 \begin{proof}
 Assume that System \eqref{eq:general_dyn_continous_time} is regularly bounded at $x_0$ with horizon $T$. Then there exists $L>0$ and $R>0$ such that for any $t\geq T$, any $s\in [t-T,t]$ and any $\xi\in \widebar{B}(x(t-T),R)$, 
 \begin{align}
    \Vert  \phi_f(s;t-T,\xi,u) \Vert \leq L.\label{eq:regular_bound_proof}
\end{align}
 
 According to Theorem 2.3.2 in \cite{bressan_introduction_2007},  for any $t\geq T$ and $s\in [t-T,t]$, and any $\xi\in \widebar{B}(x(t-T),R) $, $\Phi_f(s;t-T,\xi,u)=M(s,t-T)$  is the solution  of the following matrix-valued linear Cauchy problem:
\begin{align*}
    \diff_s M(s,t-T)&=\diff_{x} f(\phi_f(s;t-T,\xi,u) ,u(s))M(s,t-T),\\
    M(t-T,t-T)&=I_{n_x}.
\end{align*}
By integrating on $[t-T,t]$ and taking the norm, one gets for any $t\geq T$ and $s\in [t-T,t]$:
\begin{align}
    \Vert M(s,t-T)\Vert &\leq \Vert  M(t-T,t-T)\Vert+ &\int_{t-T}^t \Vert\diff_{x} f(\phi_f(s;t-T,\xi,u) ,u(s))\Vert \Vert M(s,t-T)\Vert \mathrm{d}s.
\end{align}
By assumption, $\diff_{x}f$ is continuous. Thus, from Hypothesis \ref{as:U_compact} and \eqref{eq:regular_bound_proof}, there exists $\sigma_1>0$ such that  for any $t\geq T$, $s\in [t-T,t]$, and any $\xi\in \widebar{B}(x(t-T),R)$, $\Vert\diff_{x} f(\phi_f(s;t-T,\xi,u) ,u(s))\Vert \leq\sigma_1$. This leads  for any $t\geq T$, $s\in [t-T,t]$, and any $\xi\in \widebar{B}(x(t-T),R)$ to:
\begin{align*}
    \Vert M(s,t-T)\Vert \leq 1+\sigma_1\int_{t-T}^t \Vert M(s,t-T)\Vert \mathrm{d}s.
\end{align*}
By Gronwall Lemma, $\Vert \Phi_f(s;t-T,\xi,u)\Vert \leq L_1$ where $L_1=\exp(\sigma_1 T)>0$. One can obtain \eqref{eq:regular_bound_2_diff} using the same argument by applying Theorem 2.3.2 in \cite{bressan_introduction_2007} to the system represented by $(M(s,t-T),\xi)$ and combining it with \eqref{eq:regular_bound_proof} and \eqref{eq:regular_bound_1_diff}. Finally, under Hypothesis \ref{as:three_times_diff},  $\diff^2_{\xi_2} \Phi_f(s;t-T,\xi,u)$ is well defined and \eqref{eq:regular_bound_3_diff} can be obtained similarly.
 \end{proof}
\section{Proof of Lemma \ref{lem:obs_gram_hess_regular}}\label{app:obs_gram_hess_regular}

\begin{proof}

{$\ref{item:lower_bound_grammian_obs_lem}\Rightarrow \ref{item:lower_bound_hess_lem}$}:
Assume that \ref{item:lower_bound_grammian_obs_lem} holds. Then, there exists $\mu'>0$ such that for any $t\geq T$:
\begin{align}
     \mathcal{C}(t,T,x(t-T),u)\succeq \mu'I_{n_x}.\label{eq:hessian_lower_bound_ref_proof}
\end{align}
We recall that for any $t\geq T$:
\begin{align}
   \mathcal{C}(t,T,x(t-T),u)= \frac{1}{2}\mathrm{d}^2_{\xi_2}l(t-T,t,x(t-T),x(t-T),u).
\end{align}

Furthermore, since System \eqref{eq:general_dyn_continous_time} is regularly bounded at $x_0$ with horizon $T$ then there exist $R'>0$ and $L>0$ such that for any $t\geq T$, any $s\in [t-T,t]$ and any $\xi\in \widebar{B}(x(t-T),R')$, 
\begin{align}
    \Vert  \phi_f(s;t-T,\xi,u) \Vert \leq L.\label{eq:regular_bound_2}
\end{align}
From Lemma \ref{lem:regular_bound_diff} and Hypothesis \ref{as:three_times_diff},  there exist $L'>0$ such that for any $t\geq T$, any $s\in [t-T,t]$ and any $\xi\in \widebar{B}(x(t-T),R')$:
\begin{align}
    \max(\Vert  \Phi_f(s;t-T,\xi,u) \Vert,
      \Vert  \diff_{\xi} \Phi_f(s;t-T,\xi,u)\Vert,
       \Vert  \diff^2_{\xi_2} \Phi_f(s;t-T,\xi,u)\Vert) \leq L'.\label{eq:regular_bound_diff_2}
\end{align}
From Assumptions \ref{as:three_times_diff}, $\mathrm{d}^2_{\xi_2}l(t-T,t,\cdot,\cdot,u)$ is continuously differentiable.  The differential of  $\mathrm{d}^2_{\xi_2}l(t-T,t,\cdot,\cdot,u)$ is denoted by  $\diff_{(\xi_1,\xi_2)}\mathrm{d}^2_{\xi_2}l(t-T,t,\cdot,\cdot,u)$. By combining Lemma \ref{lem:1_2_diff_l} with \eqref{eq:regular_bound_2}, \eqref{eq:regular_bound_diff_2} and Hypotheses \ref{as:three_times_diff} and \ref{as:U_compact}, one gets for any $0<R\leq R'$ that:
\begin{align}
L(R)=\sup_{t\geq T}\sup_{(\xi_1,\xi_2)\in (\widebar{B}(x(t-T),R))^2}\Vert \mathrm{d}_{(\xi_1,\xi_2)}\mathrm{d}^2_{\xi_2}l(t-T,t,\xi_1,\xi_2,u) \Vert<+\infty.\label{eq:third_diff_sup}
\end{align}
Additionally, for any $0<R\leq R'$, any $(\xi_1,\xi_2)\in (\widebar{B}(x(t-T),R))^2$, by combining \eqref{eq:hessian_lower_bound_ref_proof}, \eqref{eq:third_diff_sup} and the mean value theorem applied to  $\mathrm{d}^2_{\xi_2}l(t-T,t,\cdot,\cdot,u)$ between $(x(t-T),x(t-T))$ and $(\xi_1,\xi_2)$,  one gets:
\begin{align}
    \Vert \mathrm{d}^2_{\xi_2}l(t-T,t,\xi_1,\xi_2,u)-&\mathrm{d}^2_{\xi_2}l(t-T,t,x(t-T),x(t-T),u)\Vert\leq\notag\\ 
    & L(R)(\Vert \xi_1- x(t-T) \Vert^2+\Vert \xi_2- x(t-T) \Vert^2)^{\frac{1}{2}},\notag\\
      \Vert \mathrm{d}^2_{\xi_2}l(t-T,t,\xi_1,\xi_2,u)-&\mathrm{d}^2_{\xi_2}l(t-T,t,x(t-T),x(t-T),u)\Vert\leq\sqrt{2}L(R)R. \label{eq:hessian_upper_bound_mean_value_th}
\end{align}
Combining, \eqref{eq:hessian_lower_bound_ref_proof}, \eqref{eq:hessian_upper_bound_mean_value_th} and applying the reverse triangle inequality yields
$\mathrm{d}^2_{\xi_2}l(t-T,t,\xi_1,\xi_2,u)\succeq (2\mu'-\sqrt{2}L(R)R )I_{n_x}$.
Since $L(R)$ is non increasing with $R$ from \eqref{eq:third_diff_sup}, there exist $R>0$ such that $\mu=2\mu'-\sqrt{2}L(R)R>0$. Finally, this means that, there exist $T>0$, $R>0$ and $\mu>0$ such that for any $t\geq T$ $(\xi_1,\xi_2)\in (\widebar{B}(x(t-T),R))^2$ and $\mathrm{d}^2_{\xi_2}l(t-T,t,\xi_1,\xi_2,u)\succeq \mu I_{n_x}$. Hence, the result is proven.

$ \ref{item:lower_bound_hess_lem}\Rightarrow\ref{item:lower_bound_grammian_obs_lem}$:
Take $(\xi_1,\xi_2)=(x(t-T),x(t-T))$ in \eqref{eq:hessian_lower_bound_lem}.
\end{proof}

\section{Proof of Proposition \ref{prop:converse_hessian_lower_bound_weak_reg_per}}\label{app:converse_hessian_lower_bound_weak_reg_per}

\begin{proof}
Under the assumptions of the proposition and from Definition \ref{def:finite_sens} and, there exist $T>0$, $R'>0$, $L>0$ and a $\mathcal{K}$-function $\kappa$ such that for any $t\geq T$, $s\in [t-T,T]$ and $(\xi_1,\xi_2)\in (\widebar{B}(x(t-T),R'))^2$, $l(t-T,t,\xi_1,\xi_2,u)\geq\kappa(\Vert \xi_1 -\xi_2 \Vert)$,
and $\Vert  \phi_f(s;t-T,\xi,u) \Vert \leq L$. By using the same proof technique as in the proof of Proposition \ref{prop:hessian_pos_definite_weak_per} in Appendix \ref{app:hessian_pos_definite_weak_per}, one gets that there exists $\mu'>0$ such that:
$\mathrm{d}^2_{\xi_2}l(t-T,t,x(t-T),x(t-T),u)\succeq \mu' I_{n_x}$. Finally, from Lemma \ref{lem:obs_gram_hess_regular}, there exist $R>0$ and $\mu>0$ such that for any $t\geq T$, $(\xi_1,\xi_2)\in (\widebar{B}(x(t-T),R))^2$, $\mathrm{d}^2_{\xi_2}l(t-T,t,\xi_1,\xi_2,u)\succeq \mu I_{n_x}$, thus proving the result.
\end{proof}

\section{Differential of $\tilde{l}$}\label{app:xi_v_diff_l_perturbed}

\begin{lemma}\label{lem:xi_v_diff_l_perturbed}
For any $T>0$ and any $t\geq T$, $\xi\in \mathbb{R}^{n_y}$, any input trajectory $u$ and any perturbation signal $\eta=(v,w)\in \Theta_{t,T}$,  $\tilde{l}(t-T,t,x(t-T),\cdot,u,\eta)$ is continuously differentiable in $\mathbb{R}^{n_x}$  and $\diff_{\xi} \tilde{l}(t-T,t,x(t-T),\cdot,u,\cdot)$ is  continuously differentiable in $\mathbb{R}^{n_x}\times\Theta_{t,T}$. Additionally, under Hypothesis \ref{as:three_times_diff}, $\diff^2_{\xi}\tilde{l}(t-T,t,x(t-T),\cdot,u,\cdot)$ is continuously differentiable in $\mathbb{R}^{n_x}\times\Theta_{t,T}$.

The respective differentials read for any $x_0\in \mathbb{R}^{n_x}$  any $T>0$, any $t\geq T$, any $\xi \in \mathbb{R}^{n_x}$, any $\eta=(v,w)\in \Theta_{t,T}$ and any $\Delta\eta=(\Delta v,\Delta w)\in \Theta_{t,T}$:
\begin{align}
    &\diff_{\xi} \tilde{l}(t-T,t,\xi,u,\eta)=2\int_{t-T}^{t}(h(\hat{x}(s,\xi),u(s))-h(\tilde{x}(s,w),u(s))-v(s))^T H\Phi_f(s,\xi)\mathrm{d}s,\label{eq:1_xi_diff_l_perturbed} 
\end{align}
where $\hat{x}(s,\xi)=\phi_f(s;t-T,\xi,u)$, $H\Phi_f(s,\xi)=H(\hat{x}(s,\xi),u(s))\Phi_f(s;t-T,\xi,u)$, and $\tilde{x}(s,w)$ is defined as in \eqref{eq:general_perturbed_dyn_continous_time};
\begin{align}
    &\diff^2_{\xi}\tilde{l}(t-T,t,\xi,u,\eta) = 2\mathcal{C}(t,T,\xi,u)+2\widetilde{\mathcal{R}}(t,T,x(t-T),\xi,u,\eta),\label{eq:2_xi_xi_diff_l_perturbed}
\intertext{where for any $(\Delta \xi,\Delta' \xi) \in (\mathbb{R}^{n_x})^2$:} 
   &\Delta \xi^T\widetilde{\mathcal{R}}(t,T,x(t-T),\xi,u,\eta)\Delta' \xi=\notag\\
    &\int_{t-T}^{t}(h(\hat{x}(s,\xi),u(s))-h(\tilde{x}(s,w),u(s))-v(s))^T(\diff_{\xi}H\Phi_f(s,\xi)\Delta' \xi)\cdot \Delta\xi \mathrm{d}s.
\end{align}
The differential of $\diff_{\xi} \tilde{l}$ with respect to $\eta$ reads: 
\begin{align}
&\diff_{\eta}\diff_{\xi} \tilde{l}(t-T,t,\xi,u,\eta)\Delta \eta=\notag \\
&\diff_{v}\diff_{\xi} \tilde{l}(t-T,t,\xi,u,\eta)\Delta v +\diff_{w}\diff_{\xi} \tilde{l}(t-T,t,\xi,u,\eta)\Delta w,\\
    &\diff_{v}\diff_{\xi} \tilde{l}(t-T,t,\xi,u,\eta)\cdot\Delta v=2\int_{t-T}^{t} H\Phi_f^T(s,\xi) \Delta v(s) \mathrm{d}s,\label{eq:2_v_xi_diff_l_perturbed}\\
    &\diff_{w}\diff_{\xi} \tilde{l}(t-T,t,\xi,u,\eta)\cdot\Delta w=2\int_{t-T}^{t} H\Phi_f^T(s,\xi) (H(\tilde{x}(s,w),u(s))\diff_{w}\tilde{x}(s,w) \Delta w) \mathrm{d}s.\label{eq:2_w_xi_diff_l_perturbed}
\end{align}
\end{lemma}
\begin{proof} 
Equations \eqref{eq:1_xi_diff_l_perturbed} and \eqref{eq:2_xi_xi_diff_l_perturbed} can be obtained in the same way as \eqref{eq:1_diff_l} and \eqref{eq:2_diff_l} in Lemma \ref{lem:1_2_diff_l}. To get \eqref{eq:2_v_xi_diff_l_perturbed},first notice from $\eqref{eq:1_xi_diff_l_perturbed}$ that $\diff_{\xi} \tilde{l}(t-T,t,x(t-T),\xi,u,v)$ is affine in $v$. Secondly, note that for any $T>0$, any $t\geq T$, $\xi \in \mathbb{R}^{n_x}$ and any  $\Delta v\in L_\infty([t-T,t],\mathbb{R}^{n_y})$:
\begin{align}
    \left\Vert 2\int_{t-T}^{t} H\Phi_f^T(s,\xi) \Delta v(s) \mathrm{d}s \right\Vert &\leq K_1(t,\xi)  \Vert \Delta v\Vert_{[t-T,t],\infty},\label{eq:continuity_v_diff_l_perturbed}\\
\end{align}
where $ 0\leq K_1(t,\xi)=2T\sup_{s\in[t-T,t]} \Vert H\Phi_f^T(s,\xi) \Vert <+\infty\\$

From \eqref{eq:continuity_v_diff_l_perturbed}, one gets that $J(t,\xi):\Delta v \rightarrow   2\int_{t-T}^{t} H\Phi_f^T(s,\xi) \Delta v(s) \mathrm{d}s$ is continuous on $L_\infty([t-T,t],\mathbb{R}^{n_y})$ and \eqref{eq:2_v_xi_diff_l_perturbed} follows. Furthermore, concerning the continuity of $J(t,\cdot)$, fix $T>0$, $t\geq T$. For any $(\xi,\Delta \xi) \in (\mathbb{R}^{n_x})^2$ and any  $ v\in L_\infty([t-T,t],\mathbb{R}^{n_y})$ such that $\Vert v\Vert_{[t-T,t],\infty}=1$, $\vert J(t,\xi +\Delta \xi)\cdot v-J(t,\xi)\cdot v\vert\leq \int_{t-T}^{t} \Vert H\Phi_f^T(s,\xi+\Delta \xi)-H\Phi_f^T(s,\xi)\Vert \mathrm{d}s$. By denoting $\Vert \cdot \Vert$, the operator norm of bounded linear form on $L_\infty([t-T,t],\mathbb{R}^{n_y})$, one gets that:
    \begin{align}
     \Vert J(t,\xi +\Delta \xi)-J(t,\xi)\Vert\leq& \int_{t-T}^{t} \Vert H\Phi_f^T(s,\xi+\Delta \xi)-H\Phi_f^T(s,\xi)\Vert \mathrm{d}s.\label{eq:cont_J}
\end{align}
Consequently, $\lim_{\Delta_\xi \rightarrow 0} \Vert J(t,\xi +\Delta \xi)-J(t,\xi)\Vert=0$, by the theorem of continuity under integral and continuity of $H\Phi_f^T(s,\cdot)$. This means that $J(t,\cdot)$ and thus $(\xi,v)\rightarrow J(t,\xi)$ are continuous. This further proves that $\diff_{v}\diff_{\xi} \tilde{l}(t-T,t,x(t-T),\cdot,u,\cdot)$ is continuous on $\mathbb{R}^{n_x}\times \Theta_{t,T}$. 
One can notice that, in \eqref{eq:1_xi_diff_l_perturbed}, $\diff_{\xi} \tilde{l}(t-T,t,\xi,u,\eta)$ only depends on $w$ through $\tilde{x}(s,w)$. Then, by removing constant terms with respect to w in \eqref{eq:1_xi_diff_l_perturbed} and adapting the differentiation rule under the integral sign on Banach space in Example 2.4.16 of \cite{abraham_manifolds_1988} to the case of integrands that are only  piece-wise continuous in time, one gets \eqref{eq:2_w_xi_diff_l_perturbed} and for any $t\geq T$, $\xi \in \mathbb{R}^{n_x}$ and any  $\Delta v\in L_\infty([t-T,t],\mathbb{R}^{n_y})$:
\begin{align}
    \left\Vert \int_{t-T}^{t} H\Phi_f^T(s,\xi) (H(\tilde{x}(s,w),u(s))\diff_{w}\tilde{x}(s,w) \Delta w) \mathrm{d}s \right\Vert &\leq K_2(t,\xi,w)  \Vert \Delta w\Vert_{[0,t],\infty}\label{eq:continuity_w_diff_l_perturbed}
\end{align}
where $ 0\leq K_2(t,\xi,w)=2T\sup_{s\in[t-T,t]} \Vert H\Phi_f^T(s,\xi)\Vert \Vert H(\tilde{x}(s,w),u(s)) \Vert \Vert\diff_{w}\tilde{x}(s,w) \Vert.$

As far as  the continuity of $\diff^2_{\xi}\tilde{l}(t-T,t,x(t-T),\cdot,u,\cdot)$ is concerned, for any  $(\xi,\Delta \xi, \Delta \xi') \in (\mathbb{R}^{n_x})^3$ and $\eta=(v,w) \in (\Theta)$, one has:

\begin{align}
    \diff^2_{\xi}\tilde{l}(\xi,\eta) =2\mathcal{C}(\xi)+\widetilde{\mathcal{R}}_1(\xi,v)+\widetilde{\mathcal{R}}_2(\xi,w),\label{eq:2_xi_xi_diff_cont_perturbed}
\end{align}
where:
\begin{align*}
    {\Delta \xi'}^T\widetilde{\mathcal{R}}_1(\xi,v)\Delta \xi&=-\int_{t-T}^{t}v(s)^T(\diff_{\xi}H\Phi_f(s,\xi)\Delta' \xi)\cdot \Delta\xi \mathrm{d}s,\\
    {\Delta \xi'}^T\widetilde{\mathcal{R}}_2(\xi,w)\Delta \xi&=\int_{t-T}^{t}(h(\hat{x}(s,\xi),u(s))-h(\tilde{x}(s,w),u(s)))^T(\diff_{\xi}H\Phi_f(s,\xi)\Delta' \xi)\cdot \Delta\xi \mathrm{d}s,
\end{align*}
and unnecessary dependencies have been removed. Besides, $\mathcal{C}(t,T,\cdot,u) $ is continuous on $\mathbb{R}^{n_x}$  from Lemma \ref{lem:1_2_diff_l}. Moreover, from Lemma \ref{lem:diff_process_noise}, and Example 2.4.16 of \cite{abraham_manifolds_1988}, $\widetilde{\mathcal{R}}_2$ is continuous on $\mathbb{R}^{n_x} \times L_\infty([0,t],\mathbb{R}^{n_x})$.


Note that for any  $(\xi,\delta \xi) \in (\mathbb{R}^{n_x})^2$ and $(v,\delta v)\in ( (L_\infty([t-T,t]))^2$, one gets from \eqref{eq:2_xi_xi_diff_cont_perturbed}:
\begin{align}
    \Vert\widetilde{\mathcal{R}}_1(\xi+\delta \xi,v+\delta v)-\widetilde{\mathcal{R}}_1(\xi,v)\Vert&\leq
    \Vert\widetilde{\mathcal{R}}_1(\xi+\delta \xi,v+\delta v)-\widetilde{\mathcal{R}}_1(\xi,v+\delta v)\Vert \label{eq:cont_R_tilde}\\
    &+ \Vert\widetilde{\mathcal{R}}_1(\xi,v+\delta v)-\widetilde{\mathcal{R}}_1(\xi,v)\Vert.\notag
\end{align}
From \eqref{eq:2_xi_xi_diff_cont_perturbed}, one further gets:

\begin{align*}
     \Vert\widetilde{\mathcal{R}}_1(\xi+\delta \xi,v+\delta v)-&\widetilde{\mathcal{R}}_1(\xi,v+\delta v)\Vert \\
     &\leq \left(\int_{t-T}^{t}\Vert\diff_{\xi}H\Phi_f(s,\xi+\delta\xi)-\diff_{\xi}H\Phi_f(s,\xi)\Vert  \mathrm{d}s \right ) \Vert v+\delta v \Vert_{[t-T,t]},
     \end{align*}
 \begin{align*}
      \Vert\widetilde{\mathcal{R}}_1(\xi,v+\delta v)-\widetilde{\mathcal{R}}_1(\xi,v)\Vert \leq \left(\int_{t-T}^{t}\Vert\diff_{\xi}H\Phi_f(s,\xi)\Vert  \mathrm{d}s \right ) \Vert\delta v \Vert_{[t-T,t]},
\end{align*}
leading to  $\lim_{(\delta_\xi,\delta_v) \rightarrow 0}\Vert\widetilde{\mathcal{R}}_1(\xi+\delta \xi,v+\delta v)-\widetilde{\mathcal{R}}_1(\xi,v)\Vert=0$ by continuity of  $\diff_{\xi}H\Phi_f(s,\cdot)$ which proves that $\mathcal{R}_1$ is continuous on $\mathbb{R}^{n_x} \times L_\infty([t-T,t])$. Furthermore, $\mathcal{C}$, $\mathcal{R}_1$ and $\mathcal{R}_2$ are continuous on $\mathbb{R}^{n_x}\times \Theta_{t,T}$ when seen as function of $(\xi,\eta)$ which implies $\diff^2_{\xi}\tilde{l}(t-T,t,x(t-T),\cdot,u,\cdot)$ is continuous on $\mathbb{R}^{n_x}\times  \Theta_{t,T}$. 
Finally under Assumption \ref{as:three_times_diff}, $H\Phi_f(s,\cdot)$ is twice continuously differentiable and  one can show, from \eqref{eq:2_xi_xi_diff_cont_perturbed}, by reproducing analogous arguments that  $\diff^2_{\xi}\tilde{l}(t-T,t,x(t-T),\cdot,u,\cdot)$ is continuously differentiable on $\mathbb{R}^{n_x}\times \Theta_{t,T}$. 

\end{proof}

\section{Proof of Theorem \ref{th_weak_per_robustness}}\label{app:weak_per_robustness}

\begin{proof}

 


   Assume that $u$ is a weakly persistent input trajectory at $x_0$ and all the associated $\mathcal{K}$-functions $\kappa_t$ have finite sensitivity. Then, by Corollary \ref{cor:grammian_pos_definite_weak_per}, there exists $T>0$ such that for any $t\geq T$
\begin{align}
     \mathcal{C}(t,T,x(t-T),u)\succ 0. 
\end{align}
In the sequel, we denote by $\mu_t$ the smallest eigenvalue of $\mathcal{C}(t,T,x(t-T),u)$: From Lemma \eqref{lem:xi_v_diff_l_perturbed}, one can see that  $\diff^2_{\xi}\tilde{l}(t-T,t,x(t-T),\cdot,u,\cdot)$ is continuous on $\mathbb{R}^{n_x}\times \Theta_{t,T} $ and that for any $t\geq T$:
\begin{align*}
    \diff^2_{\xi}\tilde{l}(t-T,t,x(t-T),u,0)=2 \mathcal{C}(t,T,x(t-T),u)\succeq 2\mu_t.
\end{align*}
Therefore, by continuity, for any $t\geq T$ there exist $\nu^{(1)}_t>0$ and $R_t>0$ such that for any $\xi\in \widebar{B}(x(t-T),R_t)$ and any $\eta \in B_{t,\infty}(0,\nu^{(1)}_t>0)$:
\begin{align}
    \diff^2_{\xi}\tilde{l}(t-T,t,\xi,u,\eta)\succeq \mu_t\succ 0.\label{eq:hess_pos_def_proof_robust}
\end{align}
From Lemma \ref{lem:cost_MHE_state} and \ref{lem:xi_v_diff_l_perturbed}, one gets that  for any $t\geq T$:
\begin{align}
    \diff_{\xi} \tilde{l}(t-T,t,x(t-T),u,0)=\diff_{\xi_2}{l}(t-T,t,x(t-T),x(t-T),u)=0.
\end{align}
Furthermore, Lemma \ref{lem:xi_v_diff_l_perturbed} shows that $ \diff_{\xi} \tilde{l}(t-T,t,\cdot,u,\cdot)$ is continuously differentiable on $\mathbb{R}^{n_x}\times \Theta_{t,T}$ and \eqref{eq:hess_pos_def_proof_robust} proves that $\diff^2_{\xi}\tilde{l}(t-T,t,x(t-T),u,0)$ is invertible. The Implicit Function Theorem on Banach spaces, see Theorem 3.13 in \cite{pathak_introduction_2018}, states that for any $t\geq T$, there exist $\nu^{(2)}_t>0$, $0<R'_t\leq R_t$ and a unique continuously differentiable function $\xi^*_t: B_{t,\infty}(0,\nu^{(2)}_t) \rightarrow \widebar{B}(x(t-T),R'_t)$ such that $\xi^*(0)=x(t-T)$ and for any $\eta\in B_{t,\infty}(0,\nu^{(2)}_t)$, $\xi^*_t(\eta) \in \widebar{B}(x(t-T),R'_t)$ and:
\begin{align}
     \diff_{\xi} \tilde{l}(t-T,t,\xi^*_t(\eta),u,\eta)=0.\label{eq:implicit_first_order_condition}
\end{align}

Set $\nu_t=\min(\nu^{(1)}_t\nu^{(2)}_t)$. Then, the differential of $\xi^*_t$  reads  for any $\eta\in B_{t,\infty}(0,\nu_t)$:
\begin{align}
    \diff_\eta \xi^*_t(\eta)=(\diff^2_{\xi}\tilde{l}(t-T,t,\xi^*_t(\eta),u,\eta))^{-1}\diff_{\eta}\diff_{\xi} \tilde{l}(t-T,t,\xi^*_t(\eta),u,\eta), \label{eq:v_diff_solution_perturbed_weak_per}
\end{align}
where the inverse of $\diff^2_{\xi}\tilde{l}(t-T,t,x(t-T),\xi^*_t(\eta),u,\eta)$ is ensured to exist by \eqref{eq:hess_pos_def_proof_robust}.
 By combining \eqref{eq:hess_pos_def_proof_robust} and \eqref{eq:implicit_first_order_condition}, one has that, for any $t\geq T$ and any  $\eta\in B_{t,\infty}(0,\nu_t)$:
\begin{align}
    &\diff_{\xi} \tilde{l}(t-T,t,\xi^*_t(\eta),u,\eta)=0, &\diff^2_{\xi}\tilde{l}(t-T,t,\xi^*_t(\eta),u,\eta)\succ 0.\label{eq:implicit_first_second_order_condition}
\end{align}
Note that  \eqref{eq:implicit_first_second_order_condition} implies that $\xi^*_t(\eta)$ is a strict local solution of Problem \ref{pb:receding_perturbed_horizon_gen_state}. Moreover, \eqref{eq:hess_pos_def_proof_robust} implies that, for any $t\geq T$, and any  $\eta\in B_{t,\infty}(0,\nu_t)$, $\tilde{l}(t-T,t,\cdot,u,\eta)$ is strictly convex on $\widebar{B}(x(t-T),R_t)$. Since $\xi^*_t(\eta) \in \widebar{B}(x(t-T),R_t)$  for any $\eta \in B_{t,\infty}(0,\nu_t)$, then $\xi^*_t(\eta)$ is the only local solution of Problem \eqref{pb:receding_perturbed_horizon_gen_state} on $\widebar{B}(x(t-T),R_t)$. Finally, to prove that \eqref{eq:error_solution_perturbed_weak_per} holds, we can combine \eqref{eq:hess_pos_def_proof_robust} and \eqref{eq:v_diff_solution_perturbed_weak_per} to get that, for any $t\geq T$ and any $\eta\in B_{t,\infty}(0,\nu_t)$:
\begin{align}
    \Vert   \diff_\eta \xi^*_t(\eta) \Vert \leq \frac{1}{2\mu_t} \Vert \diff_{\eta}\diff_{\xi} \tilde{l}(t-T,t,\xi^*_t(\eta),u,\eta)\Vert,\label{eq:bound_v_diff_solution_perturbed}
\end{align}
  From \eqref{eq:continuity_v_diff_l_perturbed} and \eqref{eq:continuity_w_diff_l_perturbed} in the proof of Lemma \ref{lem:xi_v_diff_l_perturbed}, for any $\Delta\eta=(\Delta v,\Delta w) \in \Theta_{t,T}$

  \begin{align}
      \Vert  \diff_{\eta}\diff_{\xi} \tilde{l}(t-T,t,\xi^*_t(\eta),u,\eta)\Delta \eta \Vert &\leq\Vert \label{eq:diff_l_eta_1} \diff_{v}\diff_{\xi} \tilde{l}(t-T,t,\xi^*_t(\eta),u,\eta)\Delta v \Vert \notag\\& 
      +\Vert  \diff_{w}\diff_{\xi} \tilde{l}(t-T,t,\xi^*_t(\eta),u,\eta)\Delta w \Vert\\
    \Vert  \diff_{v}\diff_{\xi} \tilde{l}(t-T,t,\xi^*_t(\eta),u,\eta)\Delta v \Vert &\leq C_{1,t} \Vert   \Delta v \Vert_{\infty,[t-T,t]},\label{eq:diff_l_eta_2}\\
         \Vert\diff_{w}\diff_{\xi} \tilde{l}(t-T,t,\xi^*_t(\eta),u,\eta)\Delta v \Vert &\leq C_{2,t}\Vert \Delta w \Vert_{\infty,[t-T,t]},\label{eq:diff_l_eta_3}
\end{align}
where for any $s\in [t-T,t]$ and any $\xi \in \mathbb{R}^{n_x}$, $H\Phi_f(s,\xi)=H(\phi_f(s;t-T,\xi,u),u(s))\Phi_f(s;t-T,\xi,u)$ and:
\begin{align*}
    C_{1,t}&= 2T\left( \sup_{s\in[t-T,t]} \sup_{\eta \in  B_{t,\infty}(0,\nu_t)} \Vert H\Phi_f^T(s,\xi^*_t(\eta)) \Vert \right), \\
    C_{2,t}&=2T\left(\sup_{s\in[t-T,t]}\sup_{\eta \in  B_{t,\infty}(0,\nu_t)} \Vert H\Phi_f^T(s,\xi^*_t(\eta))\Vert \Vert H(\tilde{x}(s,w),u(s)) \Vert \Vert\diff_{w}\tilde{x}(s,w) \Vert \right).\\
\end{align*}    
Note that $ C_{1,t} <+\infty$ since $H$, $\phi_f(\cdot;t-T,\cdot,u)$ and $\Phi_f(\cdot;t-T,\cdot,u)$ are continuous,  $u$ is assumed to be piecewise continuous, and for any $\eta \in  B_{t,\infty}(0,\nu_t), \xi^*_t(\eta) \in \widebar{B}(x(t-T),R_t)$. Besides, from Lemma \ref{lem:bounded_x_bar}  and \ref{lem:diff_process_noise}:
\begin{align*}
    \sup_{s\in [t-T,t]} \sup_{\Vert w\Vert_{\infty,[t-T,t]}\leq \nu_t } \Vert \tilde{x}(s,w) \Vert &<+\infty,\\
    \sup_{s\in [t-T,t]} \sup_{\Vert w\Vert_{\infty,[t-T,t]}\leq \nu_t } \Vert \diff_{w}\tilde{x}(s,w) \Vert &<+\infty.
\end{align*}
Thus, one has $C_{2,t} <+\infty$ since $\Vert w\Vert_{\infty,[t-T,t]} \leq \Vert \eta \Vert_{t,T} $ by definition. Finally, from \eqref{eq:diff_l_eta_1}- \eqref{eq:diff_l_eta_3}, one has for any $\Delta\eta=(\Delta v,\Delta w) \in \Theta_{t,T}$:
\begin{align*}
     \Vert  \diff_{\eta}\diff_{\xi} \tilde{l}(t-T,t,\xi^*_t(\eta),u,\eta)\Delta \eta \Vert \leq C_t\Vert\Delta\eta \Vert, 
     \intertext{where $C_t=C_{1,t}+C_{2,t}$, which leads to:}
    \Vert  \diff_{\eta}\diff_{\xi} \tilde{l}(t-T,t,\xi^*_t(\eta),u,\eta)\Vert \leq C_t, 
\end{align*}
 
Consequently, from \eqref{eq:bound_v_diff_solution_perturbed} , one gets  for any $t\geq T$ and any $v\in B_{t,\infty}(0,\nu_t)$, $\Vert \diff_\eta \xi^*_t(\eta) \Vert \leq K_t$, where $K_t=\frac{C_t}{2\mu_t}$. Further applying the mean value theorem to $\xi^*_t$ between $0$ and $\eta$ for  $\Vert \eta \Vert_{t,T}\leq \nu_t$, results in $ \Vert\xi^*_t(\eta)-x(t-T) \Vert\leq K_t \Vert \eta \Vert_{t,T}$,which proves \eqref{eq:error_solution_perturbed_weak_per}.

\end{proof}

\section{Proof of Proposition \ref{prop:uni_implicit_function_th}}\label{app:uni_implicit_function_th}
We start by stating a useful lemma to determine an upper bound on the norm of the inverse of a linear operator.  
\begin{lemma} \label{lem:inverse_operator} 
Let X and Y be two normed vector spaces. Let $A:X\rightarrow Y$ be a continuous linear operator and $c>0$. Set $R=image(A)\subset Y$. The following are equivalent:
\begin{enumerate}[label=(\roman*)]
    \item $A^{-1}:R\rightarrow X$ exists and $\Vert A^{-1}\Vert \leq \frac{1}{c}$;
    \item for any $x\in X$, $\Vert Ax\Vert\geq c\Vert x \Vert$;
\end{enumerate}
\end{lemma}
\begin{proof}

$(i)\Rightarrow (ii)$: Fix $x\in X$, and set $y=Ax$. By $(i)$, $\Vert A^{-1}y\Vert\leq \frac{1}{c} \Vert y  \Vert $. Thus,   $\Vert x \Vert\leq \frac{1}{c} \Vert Ax \Vert $ and $(ii)$ follows.

$(ii)\Rightarrow (i)$: See Section 2.7, Problem 7 in \cite{kreyszig_introductory_1978}.
\end{proof}
We can now show Proposition \ref{prop:uni_implicit_function_th}.
\begin{proof}[Proof of Proposition \ref{prop:uni_implicit_function_th} ]
Assume that \ref{as:admissible_ball}-\ref{as:condition_delta_eps_proof} hold. For any $t\in J$ and $(x,y)\in S_t$, set $G(t,x,y)=y-\Gamma_tF(t,x,y)$. From \ref{as:cont_diff_proof}, one has that for $t\in J$, $G(t,\cdot,\cdot)$ is continuously differentiable and that for any  $(x,y)\in S_t$:
\begin{align}
    \diff_y G(t,x,y)&=I_Y -\Gamma_t \diff_y F(t,x,y),\notag\\
                    &=-\Gamma_t(\diff_y F(t,x,y)-\diff_y F(t,x_{0,t},y_{0,t})),\notag\\
    \Vert \diff_y G(t,x,y)\Vert&\leq \Vert\Gamma_t \Vert \Vert \diff_y F(t,x,y)-\diff_y F(t,x_{0,t},y_{0,t})\Vert. \notag
    \intertext{From \ref{as:bounded_inverse_proof}, \ref{as:lip_non_decreasing_proof} and \ref{as:condition_delta_eps_proof}, one gets that:}
    \Vert \diff_y G(t,x,y)\Vert&\leq  Lg_1(\Vert x-x_{0,t} \Vert,\Vert y-y_{0,t}\Vert),\notag\\
     \Vert \diff_y G(t,x,y)\Vert&\leq Lg_1(\delta,\epsilon)\leq \alpha <1. \label{eq:contraction_proof}
\end{align}
Besides, for any $t\in J$ and any  $(x,y)\in S_t$:
\begin{align}
    \Vert G(t,x,y) -y_{0,t}\Vert &\leq  \Vert G(t,x,y) -G(t,x,y_{0,t})\Vert +\Vert G(t,x,y_{0,t}) -y_{0,t}\Vert, \notag\\
    \Vert G(t,x,y) -y_{0,t}\Vert &\leq  \Vert G(t,x,y) -G(t,x,y_{0,t})\Vert +\Vert \Gamma_t F(t,x,y_{0,t})\Vert.\notag
     \end{align}
    The Mean Value Theorem on Banach spaces (see Theorem 3.2 in  \cite{pathak_introduction_2018}) and \eqref{eq:contraction_proof} yield:
\begin{align}
    \Vert G(t,x,y) -y_{0,t}\Vert &\leq  \alpha \Vert y -y_{0,t}\Vert +\Vert \Gamma_t F(t,x,y_{0,t})\Vert .\notag
    \intertext{From \ref{as:bounded_inverse_proof}, \ref{as:lip_non_decreasing_2_proof} and \ref{as:condition_delta_eps_proof}:}
    \Vert G(t,x,y) -y_{0,t}\Vert &\leq  \alpha \Vert y -y_{0,t}\Vert +Lg_2(\Vert x-x_{0,t} \Vert),\notag\\
    \Vert G(t,x,y) -y_{0,t}\Vert &\leq  \alpha \epsilon +Lg_2(\delta),\notag\\
    \Vert G(t,x,y) -y_{0,t}\Vert &\leq  \alpha \epsilon +(1-\alpha)\epsilon=\epsilon.\label{eq:mapping_ball_proof} 
\end{align}

Then, \eqref{eq:contraction_proof} and \eqref{eq:mapping_ball_proof} imply that for any $t\in J$ and $x\in {B}(x_{0,t},\delta)$, $G(t,x,\cdot)$ is a contraction from $\widebar{B}(y_{0,t},\epsilon)$ to itself. From the Fixed Point Theorem on Banach spaces (see Proposition 3.1 in \cite{pathak_introduction_2018}) and for any $t\in J$,  there exists a unique continuous function $\phi_t: {B}(x_{0,t},\delta) \rightarrow \widebar{B}(y_{0,t},\epsilon)$ such that for any  $x\in {B}(x_{0,t},\delta)$:
\begin{align*}
&y_{0,t}=\phi_t(x_{0,t}), &F(t,x,\phi_t(x))=0,
\end{align*}
which proves \ref{eq:phi_ref_proof} and \ref{eq:implicit_proof}. To show that $\phi_t$ is continuously differentiable and that \ref{eq:diff_phi} holds, we first show that for  $(x,y)\in S_t$, $\diff_yF(t,x,y)$ is invertible. To do so, notice that from the reverse triangle inequality and for any $h \in Y$:
\begin{align}
    \Vert\diff_yF(t,x,y)\cdot h \Vert &\geq   \Vert\Gamma_t^{-1}\cdot h \Vert  - \Vert\diff_yF(t,x,y)\cdot h-\diff_yF(t,x_{0,t},y_{0,t})\cdot h \Vert \notag.
   \end{align}
  Using \ref{as:bounded_inverse_proof}, \ref{as:lip_non_decreasing_proof}, and Lemma \ref{lem:inverse_operator} applied to $\Gamma_t$, one gets $\Vert\diff_yF(t,x,y)\cdot h \Vert \geq \left( \frac{1}{L} -g_1(\delta,\epsilon) \right)\Vert h\Vert$. From \ref{as:condition_delta_eps_proof}, $\frac{1}{L} -g_1(\delta,\epsilon)>0$ so by Lemma \ref{lem:inverse_operator}, $\diff_yF(t,x,y)$ is invertible and for any $(x,y)\in S_t$, $\Vert(\diff_yF(t,x,y))^{-1}\Vert \leq \frac{L}{1-Lg_1(\delta,\epsilon)}$. Since $\phi_t(x)\in\widebar{B}(y_{0,t},\epsilon)$, then $\Vert(\diff_yF(t,x,\phi_t(x)))^{-1}\Vert \leq \frac{L}{1-Lg_1(\delta,\epsilon)}$,
and the rest of the proof follows from that of Theorem 3.13 in \cite{pathak_introduction_2018}.
\end{proof}

\section{Proof of Theorem \ref{th_weak_reg_per_robustness}}\label{app:th_weak_reg_per_robustness}
\begin{proof}

Assume that  Hypotheses \ref{as:three_times_diff},\ref{as:U_compact} and \ref{as:bounded_diff_x_tilde_w} hold and that $u$ is a weakly regularly persistent input trajectory at $x_0$ with an associated $\mathcal{K}$-function $\kappa$ that has finite sensitivity and an associated time horizon $T$ such that System \eqref{eq:general_dyn_continous_time} is regularly bounded at $x_0$ with horizon $T$. Then, by Corollary \ref{cor:grammian_bound_definite_weak_reg_per}, there exists $\mu>0$ such that for any $t\geq T$:
\begin{align}
     \mathcal{C}(t,T,x(t-T),u)\succeq \frac{\mu}{2} I_{n_x}. \label{eq:weak_reg_per_lower_bound_gram_proof}
\end{align}
Furthermore, from Hypothesis \ref{as:three_times_diff} and by Lemma \ref{lem:regular_bound_diff}, there exist $R'>0$, $L'>0$ such that for any $t\geq T$, any $s\in [t-T,t]$ and any $\xi\in \widebar{B}(x(t-T),R')$, 
\begin{align}
   \max(\Vert  \phi_f(s;t-T,\xi,u) \Vert, \Vert  \Phi_f(s;t-T,\xi,u) \Vert, 
      \Vert  \diff_{\xi} \Phi_f(s;t-T,\xi,u)\Vert ,
       \Vert  \diff^2_{\xi} \Phi_f(s;t-T,\xi,u)\Vert)  \leq L'.\label{eq:regular_bound_diff_proof}
\end{align}

Note that since, $x(t)=\phi_f(s;t-T,x(t-T),u)$ then the reference trajectory is bounded meaning that:
   \begin{align}
        \sup_{t\geq 0}  \Vert {x}(t)\Vert <+\infty\label{eq:bounded_reference}.
   \end{align}

Thus, from Hypothesis \ref{as:bounded_diff_x_tilde_w}, \eqref{eq:bounded_reference} and  Lemma \ref{lem:boubnde_x_tilde_uniform}, there exists $\nu'>0$ such that for any $0<\nu <\nu'$:
      \begin{align}
       \sup_{t\geq 0}  \sup_{\Vert w\Vert_{\infty}\leq \nu } \Vert \tilde{x}(t,w) \Vert &<+\infty,\label{eq:bounded_x_perturbed_proof}\\
       \sup_{t\geq 0}  \sup_{\Vert w\Vert_{\infty}\leq \nu } \Vert \diff_w\tilde{x}(t,w) \Vert &<+\infty.\label{eq:bounded_diff_w_x_perturbed_proof}
  \end{align}

From Lemma \ref{lem:xi_v_diff_l_perturbed}, for any $t\geq T$, any $\xi\in\mathbb{R}^{n_x}$ and any $\eta\in \Theta $, 
\begin{align}
    \diff^2_{\xi}\tilde{l}(t-T,t,\xi,u,\eta) &= 2\mathcal{C}(t,T,\xi,u)+2\widetilde{\mathcal{R}}(t,T,x(t-T),\xi,u,\eta),\label{eq:diff_l_tile_proof_th}
    \intertext{where for any $(\Delta \xi,\Delta' \xi) \in (\mathbb{R}^{n_x})^2$:} 
   \Delta \xi^T\widetilde{\mathcal{R}}(t,T,x(t-T),\xi,u,\eta)\Delta' \xi&=\notag\\
    \int_{t-T}^{t}(h(\hat{x}(s,\xi),u(s))&-h(\tilde{x}(s,w),u(s))-v(s))^T(\mathrm{}{\xi}H\Phi_f(s,\xi)\Delta' \xi)\cdot \Delta\xi \mathrm{d}s.\notag
\end{align}
Moreover, From Hypotheses \ref{as:three_times_diff}, and Lemma \ref{lem:diff_process_noise}$,  \diff^2_{\xi}\tilde{l}(t-T,t,x(t-T),\cdot,u,\cdot)$ is continuously differentiable. We denote $\diff_{(\xi,\eta)}\diff^2_{\xi}\tilde{l}(t-T,t,\xi,u,\eta)$ , the differential of $ \diff^2_{\xi}\tilde{l}(t-T,t,x(t-T),\cdot,u,\cdot)$ at $(\xi,\eta)$ where $ \mathbb{R}^{n_x}\times \Theta$ is equipped with the norm $\Vert (\xi,\eta)\Vert=\Vert \xi\Vert +\Vert \eta \Vert$.

Therefore, from Hypothesis \ref{as:U_compact}, \eqref{eq:regular_bound_diff_proof},  \eqref{eq:bounded_x_perturbed_proof}, \eqref{eq:bounded_diff_w_x_perturbed_proof} and by expanding $\diff_{(\xi,\eta)}\diff^2_{\xi}\tilde{l}(t-T,t,\xi,u,\eta)$ from \eqref{eq:diff_l_tile_proof_th}, one gets that, for any $0<\nu<\nu'$ and $0<R<R'$:
\begin{align}
a_1(\nu,R):=\sup_{t\geq T}\sup_{\Vert \eta \Vert\leq \nu} \sup_{\xi \in\widebar{B}(x(t-T),R) }   \Vert \diff_{(\xi,\eta)}\diff^2_{\xi}\tilde{l}(t-T,t,\xi,u,\eta)\Vert <+\infty.
\end{align}
From the Mean Value Theorem one gets that, for any $t\geq T $, any $\xi\in \widebar{B}(x(t-T),R) $ and any $\eta\in  B_{\infty}(0,\nu)$:
\begin{align}
    &\Vert \diff^2_{\xi}\tilde{l}(t-T,t,\xi,u,\eta) -  2\mathcal{C}(t,T,x(t-T),u)\Vert\leq g_1(\Vert \xi-x(t-T)\Vert,\Vert \eta \Vert),\label{eq:lip_constant_xi_v_xi_xi_diff_4}
\end{align}
where for any $\delta>0$ and $\epsilon>0$, $g_1(\delta,\epsilon)=a_1(\nu,R)(\delta+\epsilon)$.
Besides, from \eqref{eq:2_v_xi_diff_l_perturbed} and \eqref{eq:continuity_v_diff_l_perturbed}, one has for any $0<R<R'$, any $t\geq T$ , any $\xi \in \widebar{B}(x(t-T),R)$ and any  $\eta\in \Theta$:
\begin{align*}
     \Vert\diff_{v}\diff_{\xi} \tilde{l}(t-T,t,\xi ,u,\eta)\Vert&\leq2\int_{t-T}^{t} \Vert H\Phi_f^T(s,\xi )\Vert  \mathrm{d}s,\label{eq:bound_diff_v_xi_l_perturbed}\notag
\end{align*}
Similarly from \eqref{eq:2_w_xi_diff_l_perturbed} and \eqref{eq:continuity_w_diff_l_perturbed}, one has for any $0<R<R'$, any $t\geq T$ , any $\xi \in \widebar{B}(x(t-T),R)$ and any  $\eta\in \Theta$:

\begin{align*}
    \Vert \diff_{w}\diff_{\xi} \tilde{l}(t-T,t,\xi,u,\eta)\Vert&\leq 2\int_{t-T}^{t} \Vert H\Phi_f^T(s,\xi) (H(\tilde{x}(s,w),u(s))\diff_{w}\tilde{x}(s,w)) \Vert \mathrm{d}s.
\end{align*}
 Lemma \ref{lem:xi_v_diff_l_perturbed} yields:
\begin{align}
    \Vert \diff_{\eta}\diff_{\xi} \tilde{l}(t-T,t,\xi,u,\eta)\Vert&\leq\label{eq:diff_l_pertubed_xi_eta}\\
    2\int_{t-T}^{t} &\Vert H\Phi_f^T(s,\xi)\Vert (1+\Vert (H(\tilde{x}(s,w),u(s))\diff_{w}\tilde{x}(s,w)) \Vert) \mathrm{d}s.\notag
    \intertext{In particular,}
       \Vert \diff_{\eta}\diff_{\xi} \tilde{l}(t-T,t,x(t-T),u,\eta)\Vert&\leq\notag\\
    2\int_{t-T}^{t} \Vert H\Phi_f^T(s,x(t-T))&\Vert (1+\Vert (H(\tilde{x}(s,w),u(s))\diff_{w}\tilde{x}(s,w)) \Vert) \mathrm{d}s,\notag
\end{align}

From \eqref{eq:regular_bound_diff_proof}, Hypothesis \ref{as:U_compact}, \eqref{eq:bounded_x_perturbed_proof} and \eqref{eq:bounded_diff_w_x_perturbed_proof},   one gets for any $\nu>0$ that:

\begin{align*}
    a_2(\nu):=2\sup_{t\geq T }\sup_{\Vert w \Vert_{\infty}\leq \nu}\int_{t-T}^{t} &\Vert H\Phi_f^T(s,x(t-T))\Vert (1+\Vert (H(\tilde{x}(s,w),u(s))\diff_{w}\tilde{x}(s,w)) \Vert) \mathrm{d}s<+\infty.
\end{align*}
 Thus, by recalling that  for any $t\geq T$:
\begin{align}
    \diff_{\xi} \tilde{l}(t-T,t,x(t-T),u,0)&=0,\label{eq:first_order_cond_proof}
\intertext{and one has, from the Mean Value theorem, for any $\nu>0$ and any $\eta\in  B_\infty(0,\nu)$:}
     \Vert\diff_{\xi} \tilde{l}(t-T,t,x(t-T),u,\eta)\Vert&\leq g_2(\Vert \eta\Vert),\label{eq:lip_constant_v_xi_diff_2}
\end{align}
where for any $\delta>0$, $g_2(\delta)=a_2(\nu)\delta$. We now fix $0<R<R'$, $\nu>0$ and $0<\alpha<1$, and assume that:
\begin{align}
       &\frac{g_1(\nu,R)}{\mu}\leq \alpha<1,&\frac{g_2(\nu)}{\mu}\leq R (1-\alpha).\label{eq:condition_radius_proof}
\end{align}
Applying Proposition \ref{prop:uni_implicit_function_th} with  $J=[T,+\infty[$, $X=\Theta$, $Y=\mathbb{R}^{n_x}$,  $Z=\mathbb{R}^{n_x}$ and $\Omega=\Theta\times\mathbb{R}^{n_x}$;
  one obtains that for any $t\geq 0$, $y_{0,t}=x(t-T)$ and $x_{0,t}=0$ and  $F=\diff_{\xi}\tilde{l}$;
 $\delta=\nu$, $\epsilon=R$ and $L=\frac{1}{\mu}$. Note that in this case, \ref{as:admissible_ball} is clear, \ref{as:equation_ref_proof} holds thanks to \eqref{eq:first_order_cond_proof}, \ref{as:cont_diff_proof} holds thanks to Lemma \ref{lem:xi_v_diff_l_perturbed} and \ref{as:bounded_inverse_proof} holds from \eqref{eq:weak_reg_per_lower_bound_gram_proof} with $\Gamma_t=\left( 2\mathcal{C}(t,T,x(t-T),u)\right)^{-1}$. Note also that \ref{as:lip_non_decreasing_proof} and \ref{as:lip_non_decreasing_2_proof} hold from \eqref{eq:lip_constant_xi_v_xi_xi_diff_4} and \eqref{eq:lip_constant_v_xi_diff_2} with $g_1$ and $g_2$ being variable-wise non-decreasing and vanishing at 0, and \ref{as:condition_delta_eps_proof} is ensured by \eqref{eq:condition_radius_proof}.Therefore, by Proposition \ref{prop:uni_implicit_function_th}, there is a unique continuously differentiable mapping $\xi^*_t:B_{\infty}(0,\nu)\rightarrow \widebar{B}(x(t-T),R)$ such that  for any $t\geq T $, $\xi^*_t(0)=x(t-T)$ and for any $\eta\in  B_{\infty}(0,\nu)$:
\begin{align}
      \diff_{\xi} \tilde{l}(t-T,t,x(t-T),\xi^*_t(\eta),u,\eta)&=0.\label{eq:first_order_condition_sol_perturbed}
      \intertext{Additionally, $\diff^2_{\xi}\tilde{l}(t-T,t,\xi^*_t(\eta),u,\eta)$ is invertible and its inverse satisfies:}
      \Vert (\diff^2_{\xi}\tilde{l}(t-T,t,\xi^*_t(\eta),u,\eta))^{-1}\Vert &\leq \frac{1}{\mu-g_1(\nu,R)}\label{eq:inverse_xi_xi_diff_l_perturbed_sol}.
\end{align}
Also, the differential of $\xi^*_t$  for any $\eta\in  B_{\infty}(0,\nu)$ is:
      \begin{align}
          \diff\xi^*_t(\eta)=(\diff^2_{\xi}\tilde{l}(t-T,t,\xi^*_t(\eta),u,\eta))^{-1}\diff_\eta\diff_{\xi}\tilde{l}(t-T,t,\xi^*_t(\eta),u,\eta).    \label{eq:diff_xi_star}
      \end{align}
  The uniqueness in Proposition \ref{prop:uni_implicit_function_th} ensures that  for any $t\geq T $, and for any $\eta\in  B_{\infty}(0,\nu)$,  $\xi^*_t(\eta)$ is the only element of $\widebar{B}(x(t-T),R)$ satisfying \eqref{eq:first_order_condition_sol_perturbed} and thus the only local solution of \eqref{pb:receding_perturbed_horizon_gen_state} in  $\widebar{B}(x(t-T),R)$. Furthermore, from Lemma \ref{lem:inverse_operator}, \eqref{eq:inverse_xi_xi_diff_l_perturbed_sol} ensures that  for any $t\geq T $, and for any $\eta\in  B_{\infty}(0,\nu)$:
  \begin{align}
      \diff^2_{\xi}\tilde{l}(t-T,t,\xi^*_t(\eta),u,\eta)\succeq (\mu-g_1(\nu,R))I_{n_x},
  \end{align}
 which implies that $\xi^*_t(\eta)$ is a strict local solution of \eqref{pb:receding_perturbed_horizon_gen_state}.
  
  Moreover,  \eqref{eq:diff_xi_star} and \eqref{eq:inverse_xi_xi_diff_l_perturbed_sol} also imply that for any $t\geq T $, and any $\eta\in  B_{\infty}(0,\nu)$:
  \begin{align}
      \Vert \diff\xi^*_t(\eta)\Vert \leq \frac{1}{\mu-g_1(\nu,R)} \Vert \diff_\eta\diff_{\xi}\tilde{l}(t-T,t,\xi^*_t(\eta),u,\eta)  \Vert. \label{eq:bound_diff_xi_star}
  \end{align}
  Hypothesis \ref{as:U_compact} and  \eqref{eq:regular_bound_diff_proof} and Lemma \ref{lem:bounded_x_bar} and \ref{lem:diff_process_noise} lead to  
 \begin{align*}
      &g_3(\nu,R):=\\
      &2\sup_{t\geq T }\sup_{\eta\in  B_{\infty}(0,\nu)}\sup_{\xi\in \widebar{B}(x(t-T),R)} \int_{t-T}^{t} \Vert H\Phi_f^T(s,\xi)\Vert(1+\Vert (H(\tilde{x}(s,w),u(s))\diff_{w}\tilde{x}(s,w)) \Vert) \mathrm{d}s<+\infty,
 \end{align*}
     and from \eqref{eq:diff_l_pertubed_xi_eta} and \eqref{eq:bound_diff_xi_star}, one has $  \Vert \diff\xi^*_t(\eta)\Vert \leq \frac{g_3(\nu,R)}{\mu-g_1(\nu,R)}$.
Finally, by the Mean Value Theorem, one gets that, for any $t\geq T $, and any $\eta\in B_{\infty}(0,\nu)$, $\Vert \xi_t^*(\eta)-x(t-T) \Vert \leq \frac{g_3(R)}{\mu-g_1(\nu,R)} \Vert\eta \Vert_{\infty}$ and the result is proven since $g_3$ is variable-wise non decreasing.
\end{proof}

\section{Proof of Proposition \ref{prop:ex_per}}\label{app:ex_per}

\begin{proof}

 \begin{enumerate}[label=\arabic*., labelindent=0pt]
 
    \item \textit{Radial constant input trajectory}
    Let $\sigma \in \mathbb{R}$. From \eqref{eq:flow_single_int} and \eqref{eq:ex_input_cst}, one gets for any $\xi \in \mathbb{R}^{n_x}$, any $T>0$, any $t\geq T$ and any $s\in [t-T,t]$:
     \begin{align}
         \phi(s;t-T,\xi ,u_{cst})&=\xi+\sigma(s-t+T)(\ell-x_0)\label{eq:ex_state_cst}.
    \end{align}
       Thus, for any $\xi \in \ell+\mathbb{R}(\ell-x_0)$, $\phi(s;t-T,\xi ,u_{cst}) \in \ell+\mathbb{R}(\ell-x_0)$ and $h(\phi(s;t-T,\xi ,u_{cst}))=\frac{\ell-x_0}{\Vert \ell -x_0 \Vert}$ which implies that for any $T>0$, any $t\geq T$ and  $\xi \in \ell+\mathbb{R}(\ell-x_0)$, $l(t,T,x_0,\xi,u_{cst})=0$.
     Since one can find vectors $\xi \in \ell+\mathbb{R}(\ell-x_0)$  arbitrarily close to $x_0$,  this implies, by Definition \ref{def:weakly_persistent_input}, that for any $\sigma\in \mathbb{R}$, $u_{cst}(\cdot,\sigma)$ is not a weakly persistent input trajectory of System \eqref{eq:dyn_1landmark_continuous_time} at $x_0$. Besides for any $s>0$ by choosing $T=s$ and $t=T$,, one gets $l(s,0,x_0,\xi,u_{cst})=0$, for any $\xi \in \ell+\mathbb{R}(\ell-x_0)$ According to Definition \ref{def:universal_input}, this also proves that for any $\sigma\in \mathbb{R}$, $u_{cst}(\cdot,\sigma)$ is not a universal input.
     
     \item  \textit{Circular input trajectory}
     Let $\omega>0$ and $r_c>0$. From \eqref{eq:flow_single_int} and \eqref{eq:ex_input_circ},  one gets for any $T>0$, any $t\geq T$ and any $s\in [t-T,t]$, that:
     \begin{align}
    &\phi(s;t-T,x_0 ,u_{circ})=\ell+ r_0\begin{bmatrix}
                                        \cos(\omega (s-t+T)+{\psi}_{t-T})\\
                                        \sin(\omega (s-t+T)+{\psi}_{t-T})
                                      \end{bmatrix},
                                      &r(s)=r_0>0,\label{eq:ex_state_circ}
     \end{align}
     where ${\psi}_{t-T}=\psi_0+\omega(t-T)$.
     For any $T>0$ and  any $t\geq T$, we denote by $\lambda_{+}^{circ}(t,T)$ and $\lambda_{-}^{circ}(t,T)$ the two eigenvalues of  $\mathcal{C}(t,T,\xi,u_{circ})$. Following straightforward but cumbersome computations, one gets $\lambda_{\pm}^{circ}(t,T)= \frac{1}{2r_0^2}\left[T\pm \frac{\vert \sin(\omega T)\vert}{\omega}\right]$. Since $\omega>0$, then $\vert \sin(\omega T)\vert< \omega T$ and for any $T>0$ and  any $t\geq T$, $ \lambda_{+}^{circ}(t,T)\geq\lambda_{-}^{circ}(t,T)>0$ and do not depend on $t$. Thus, $u_{circ}(\cdot,\omega,r_0)$ satisfies \eqref{eq:gram_bound_prop} in Corollary \ref{cor:grammian_bound_definite_weak_reg_per} for any $T>0$. Besides, from \eqref{eq:ex_state_circ}, for any $\xi \in \mathbb{R}^{n_x}$, $ \Vert  \phi(s;t-T,\xi ,u_{circ}) \Vert \leq \Vert \ell \Vert +\Vert \xi-\ell \Vert$. Thus, for any $T>0$, System \eqref{eq:dyn_1landmark_continuous_time} is regularly bounded at $x_0$ with horizon $T$. Moreover, $u_{circ}$ is valued in a compact set, satisfying Hypothesis \ref{as:U_compact} and System \ref{eq:dyn_1landmark_continuous_time} satisfies Hypothesis \ref{as:three_times_diff} as it is linear.
Therefore, by Corollary \ref{cor:grammian_bound_definite_weak_reg_per}, for any $\omega>0$ and $r_0>0$ $u_{circ}(\cdot,\omega,r_0)$ is a weakly regularly persistent input trajectory of System \ref{eq:dyn_1landmark_continuous_time} at $x_0$.
     \item  \textit{Outward spiral input trajectory}     Let $\omega>0$, $\alpha>0$ and  $r_0>0$.  From \eqref{eq:flow_single_int} and \eqref{eq:ex_input_spi},  one gets for any $T>0$, any $t\geq T$ and any $s\in [t-T,t]$, that:
     \begin{align}
   \phi(s;t-T,x_0 ,u_{spi})&=\ell+ r(t-T)\exp(\alpha(s-t+T))\begin{bmatrix}
                                        \cos(\omega (s-t+T)+{\psi}_{t-T})\\
                                        \sin(\omega (s-t+T)+{\psi}_{t-T})
                                      \end{bmatrix},\label{eq:ex_state_spi}
     \end{align}
     where $ r(t-T)=r_0 \exp(\alpha(t-T))$. For any $T>0$ and  any $t\geq T$, we denote by $\lambda_{+}^{spi}(t,T)$ and $\lambda_{-}^{spi}(t,T)$ the two eigenvalues of  $\mathcal{C}(t,T,\xi,u_{spi})$. Following again simple but cumbersome computations, $\lambda_{\pm}^{spi}(t,T)$ read for any $T>0$ and any $t\geq T$:
 \begin{align}
     \lambda_{\pm}^{spi}(t,T)=\frac{1}{4\alpha r(t-T)^2}\left[ \exp(2T\alpha)-1 \pm b(\alpha,\omega,T) \right],\label{eq:eigenvalue_spi}
 \end{align}
where $b(\alpha,\omega,T)=\frac{\alpha}{\sqrt{\alpha^2+\omega^2}}(\exp(4T\alpha)-2\exp(2T\alpha)\cos(2T\omega)+1)^{\frac{1}{2}}$. Since $\cos(2\omega  T)\geq -1$, one gets for any $T>0$ and any $t\geq T$ and from \eqref{eq:ex_state_spi}:
\begin{align}
    b(\alpha,\omega,T)&\leq \frac{\alpha}{\sqrt{\alpha^2+\omega^2}}(\exp(4T\alpha)+2\exp(2T\alpha)+1)^{\frac{1}{2}},\notag\\
     \lambda_{-}^{spi}(t,T)&\geq \frac{1}{4\alpha r(t-T)^2}\left[\left(1-\frac{\alpha}{\sqrt{\alpha^2+\omega^2}}\right)\exp(2T\alpha)-\left(1+\frac{\alpha}{\sqrt{\alpha^2+\omega^2}}\right) \right]. \label{eq:lower_bound_eigenvalue_spi}
\end{align}
Thus, from \eqref{eq:lower_bound_eigenvalue_spi},  for any $T>0$ and any $t\geq T$, if $T>\frac{1}{2\alpha}\ln\left(\frac{\sqrt{\alpha^2+\omega^2}+\alpha}{\sqrt{\alpha^2+\omega^2}-\alpha}\right)$, then $ \lambda_{-}^{spi}(t,T)>0$ and $\mathcal{C}(t,T,x_0,u_{spi})\succ 0$. Therefore, by Corollary \ref{cor:grammian_pos_definite_weak_per}, one gets that, for any $\omega>0$, any $\alpha>0$ and any $r_0>0$, $u_{spi}(\cdot,\omega,\alpha,r_0)$ is a weakly persistent input trajectory of System \eqref{eq:dyn_1landmark_continuous_time} at $x_0$. Furthermore, by \eqref{eq:ex_state_spi} and \eqref{eq:eigenvalue_spi}, $\lim_{t\rightarrow +\infty} \lambda_{+}^{spi}(t,T)=0 $ and $0\preceq   \mathcal{C}(t,T,x_0,u_{spi}) \preceq \lambda_{+}^{spi}(t,T) I_2$. This implies that $\lim_{t\rightarrow +\infty} \Vert  \mathcal{C}(t,T,\xi,u_{spi})\Vert=0$ and the second result is proven.
\end{enumerate}
\end{proof}

\bibliographystyle{siamplain}
\bibliography{bibfile}
\end{document}